\documentclass[11pt]{article}
\usepackage[a4paper,margin=1in]{geometry}
\usepackage[T1]{fontenc}
\usepackage{lmodern}
\usepackage{amsmath,amssymb,amsthm,mathtools}
\usepackage{graphicx}
\usepackage{subcaption}
\usepackage{enumitem}
\usepackage[numbers,sort&compress]{natbib}
\usepackage[unicode,colorlinks=true,linkcolor=blue,citecolor=blue,urlcolor=blue]{hyperref}
\hypersetup{
  pdftitle={A Composite Representation Theorem for Solution Mappings in Regularized Composite Optimization with Applications to Hausdorff Lipschitz Stability},
  pdfauthor={Peixuan Yang, Wei Yao, Jin Zhang}
}
\numberwithin{equation}{section}

\theoremstyle{plain}
\newtheorem{theorem}{Theorem}[section]
\newtheorem{proposition}[theorem]{Proposition}

\newtheorem{corollary}[theorem]{Corollary}
\theoremstyle{definition}
\newtheorem{example}[theorem]{Example}
\newtheorem{condition}[theorem]{Condition}

\theoremstyle{remark}
\newtheorem{remark}[theorem]{Remark}

\title{A Composite Representation Theorem for Solution Mappings in\\
Regularized Composite Optimization with Applications to\\
Hausdorff Lipschitz Stability}
\author{Peixuan Yang \and Wei Yao \and Jin Zhang}
\date{}

\begin{document}
	
\maketitle
	
\begin{abstract}
We study the geometry and stability of solution mappings in regularized composite optimization, allowing multiple minimizers. We establish a composite representation that separates optimal linear images from their corresponding primal solutions through an image mapping and a generalized basis pursuit mapping. Under a fiber-attainment condition, the image mapping admits a reduced optimization formulation; neither representation requires convexity. For convex regularizers, suitable qualification conditions yield a Fenchel-conjugate representation of the reduced value function and a geometric characterization of the generalized basis pursuit mapping through affine fibers and a fixed lifting set. Using these representations, we establish sufficient conditions for local Hausdorff Lipschitz continuity of possibly set-valued solution mappings. We further show that the primal-space sufficient conditions imply the corresponding image-space conditions, with an additional injectivity requirement arising in the former. The results cover $\ell_1$-type and structured piecewise linear-quadratic regularizers under suitable assumptions. For classical $\ell_1$ regularization, our approach recovers an existing stability result without a finite sign-pattern decomposition. Numerical examples illustrate the recovery of complete solution sets, including a case with a nonpolyhedral graph of the generalized basis pursuit mapping.
\end{abstract} 
	
	


\section{Introduction}

\subsection{Regularized composite optimization}

Many problems in high-dimensional statistics, signal processing,
and inverse problems involve recovering an unknown parameter or
signal from noisy, incomplete, or indirect observations
\cite{AdcockHansen2021,BuhlmannVanDeGeer2011,
HastieTibshiraniFriedman2009}.
When information is limited or the observation operator is
ill-conditioned, data fitting alone may yield nonunique or
unstable solutions. Regularization addresses this difficulty
by incorporating structural information into the estimation
problem \cite{EnglHankeNeubauer1996}.

Different structural assumptions lead to different regularizers.
Coordinatewise and group sparsity underlie the LASSO and group
LASSO \cite{Tibshirani1996,YuanLin2006}, while sparsity also
plays a central role in compressed sensing
\cite{CandesRombergTao2006Stable,Donoho2006}.
Total variation regularization promotes spatial structure
in imaging and inverse problems
\cite{ChambollePock2011,RudinOsherFatemi1992},
whereas sparsity under a linear transformation leads to
generalized LASSO models
\cite{AliTibshirani2019,TibshiraniTaylor2011}.

These models share a common structure: the data-fidelity term
depends on the decision variable through a linear transformation,
while the regularizer promotes desired structural properties.
Motivated by this structure, we consider the parametric
regularized composite optimization problem
\begin{equation}\label{P1}
\min_{x\in\mathbb R^n}
\left\{
F(Ax,b)+\lambda R(x)
\right\},
\qquad
(b,\lambda)\in\mathbb R^d\times\mathbb R_{++},
\end{equation}
where $A\in\mathbb R^{m\times n}$ is a fixed linear operator,
$b\in\mathbb R^d$ is a problem parameter,
$F:\mathbb R^m\times\mathbb R^d\to\mathbb R$ is a
data-fidelity or loss function,
$R:\mathbb R^n\to(-\infty,+\infty]$ is a regularizer,
and $\lambda>0$ is the regularization parameter.
The parameter $b$ may represent observed data or other
problem-specific information, with $d$ not necessarily
equal to $m$. Depending on the application, $A$ may
represent a design matrix, measurement matrix, or
forward operator.

We study the associated solution mapping
\[
S(b,\lambda)
:=
\arg\min_{x\in\mathbb R^n}
\left\{
F(Ax,b)+\lambda R(x)
\right\}.
\]
Problem~\eqref{P1} has been studied from various
perspectives, including existence, uniqueness, sensitivity,
stability, and statistical recovery guarantees.
In the convex setting, conditions for uniqueness have
been investigated in
\cite{BerkBrugiapagliaHoheisel2023,
BerkBrugiapagliaHoheisel2024,HuYaoZhang2025},
while local regularity properties of solution mappings
have also received considerable attention; see \cite{Nghia2025} and references therein. 

When multiple minimizers coexist, the geometry and
stability of the solution set become more involved,
motivating structural representations that facilitate
its analysis under parameter perturbations.


\subsection{Contributions}

The main contributions are summarized as follows.

First, we develop a composite representation of the
solution mapping $S(b,\lambda)$:
\[
S(b,\lambda)
=
\bigcup_{y\in H(b,\lambda)}\Phi(y)
=
(\Phi\circ H)(b,\lambda),
\]
where
\begin{equation}\label{e-H-definition}
H(b,\lambda)
:=
\{Ax:x\in S(b,\lambda)\}
\end{equation}
is the image mapping, and
\begin{equation}\label{e-Phi-definition}
\Phi(y)
:=
\arg\min_{x\in\mathbb R^n}
\{R(x):Ax=y\}
\end{equation}
is the generalized basis pursuit mapping.
This representation separates the optimal linear images
from the corresponding minimum-regularizer solutions
and does not require convexity.
Under a fiber-attainment condition, we further obtain
\[
H(b,\lambda)
=
\arg\min_{y\in\mathbb R^m}
\{F(y,b)+\lambda V(y)\},
\]
where
\begin{equation}\label{e-V-definition}
V(y)
:=
\inf_{x\in\mathbb R^n}
\{R(x):Ax=y\}
\end{equation}
is the generalized basis pursuit value function.
For proper, closed, convex regularizers, we also derive
a Fenchel-conjugate representation of $V$ under a
suitable qualification condition.

Second, we establish a geometric representation of
$\Phi$ for proper, closed, convex regularizers.
Under a relative-interior qualification condition,
\[
\Phi(y)=A^{-1}(y)\cap C_R,
\]
where $A^{-1}(y):=\{x:Ax=y\}$ and
$C_R:=(\partial R)^{-1}(\operatorname{rge}(A^T))$
is a fixed lifting set determined by $(A,R)$.
This representation reduces the geometric analysis
of $\Phi$ to intersections of affine fibers with
a fixed set. We use this structure to study
Hausdorff Lipschitz continuity when $C_R$ is
convex polyhedral or a finite union of convex
polyhedral sets.

Finally, we establish sufficient conditions for
local Hausdorff Lipschitz stability of possibly
set-valued solution mappings. By combining
local single-valued Lipschitz continuity of $H$
with Hausdorff Lipschitz continuity of $\Phi$,
we obtain stability results without requiring
uniqueness of the original minimizers.
We further show that the primal-space sufficient
conditions imply the corresponding image-space
conditions, with an additional injectivity
requirement arising in the primal-space
nondegeneracy condition.

These results apply to $\ell_1$-type and structured
piecewise linear-quadratic regularizers under suitable
conditions. For classical $\ell_1$ regularization,
our approach recovers a Hausdorff Lipschitz stability
result of \cite{HuYaoZhang2025} without its finite
sign-pattern decomposition. Numerical experiments
illustrate the recovery of complete solution sets,
including cases with nonpolyhedral fiber geometry.


\subsection{Related work}

The stability and sensitivity of set-valued solution
mappings are classical topics in variational analysis
and parametric optimization. Foundational results include
strong regularity of generalized equations
\cite{Robinson1980}, stability of linear and polyhedral
optimization \cite{Klatte1987,MangasarianShiau1987},
and Lipschitz stability of parametric variational
inequalities \cite{Yen1995}. These developments form
part of the broader theory of metric regularity and
perturbation analysis; see
\cite{BonnansShapiro2000,DontchevRockafellar2009,
Ioffe2017,Mordukhovich2006,Mordukhovich2018,RW1998}.

Beyond statistical recovery guarantees, regularized
optimization has been studied through active-set
identification, partial smoothness, uniqueness, and
sensitivity \cite{Lewis2002,VaiterPeyreDossalFadili2018}.
More recently, Lipschitz stability of convex regularized
optimization has been investigated using geometric
conditions \cite{Nghia2025},
$\mathcal C^2$-cone reducibility
\cite{CuiHoheiselNghiaSun2026},
and generalized Hessian and monotone
generalized-equation techniques
\cite{SmulanskyHoheiselNghia2026}.

The image-space reduction considered here is related
to infimal postcomposition and partial minimization
in convex analysis. The value function
$V(y)=\inf_{Ax=y}R(x)$ allows the original problem
to be reduced to the image space, with minimizers
lifted through the corresponding fibers.
Related formulations and lifting results have been
studied in the convex setting; see, for example,
\cite{BricenoAriasPustelnik2024}.
These connections provide a variational perspective
on the image-space representation used in this paper.

The works most closely related to our stability
analysis concern LASSO-type solution mappings.
Local sensitivity and Lipschitz properties of the
classical and square-root LASSO were studied in
\cite{BerkBrugiapagliaHoheisel2023,
BerkBrugiapagliaHoheisel2024}.
For convex data-fidelity functions with $\ell_1$
regularization, \cite{HuYaoZhang2025} developed
a decomposition approach based on a parameterized
generalized equation for the image mapping and a
finite polyhedral decomposition of the primal
solution mapping. Related Lipschitz results for
a broader class of $\ell_1$-regularized solution
mappings were established in \cite{meng2024lip}.
In the $\ell_1$ setting, our approach is closely
related to \cite{HuYaoZhang2025}, but uses the
generalized basis pursuit mapping rather than
a sign-pattern decomposition. 

	
\textbf{Organization.}
Section~\ref{sec2} establishes the composite representation
theorem and characterizes its constituent mappings.
Section~\ref{sec3} develops sufficient conditions for
local Hausdorff Lipschitz stability in the convex setting.
Section~\ref{sec4} presents numerical illustrations of
solution-set recovery, and Section~\ref{sec5} concludes
the paper.


\textbf{Notation.}
Throughout, $\mathbb R_+$ and $\mathbb R_{++}$ denote
the nonnegative and positive real numbers, respectively,
and $[n]:=\{1,\ldots,n\}$. For $t\in\mathbb R$,
$\operatorname{sgn}(t)$ denotes its sign.
For $I\subseteq[n]$, let $I^c:=[n]\setminus I$;
$x_I$ and $A_I$ denote the subvector and submatrix
indexed by $I$, respectively.

All spaces are finite-dimensional Euclidean spaces with
inner product $\langle\cdot,\cdot\rangle$ and norm
$\|\cdot\|$, also used for the operator norm.
For a linear mapping $A$, $\operatorname{rge}A$,
$\ker A$, and $A^T$ denote its range, kernel, and adjoint.
For a subspace $K$, $K^\perp$ denotes its orthogonal
complement. We write $A^{-1}(y):=\{x:Ax=y\}$ for
$y\in\operatorname{rge}A$, and $A^\dagger$ for the
Moore--Penrose pseudoinverse. 

For a set $C$, $\operatorname{ri}C$ and
$\operatorname{par}C$ denote its relative interior
and parallel subspace; $\delta_C$ and $N_C$
denote its indicator function and normal cone.
We write $B(\bar x,\delta)$ for the open ball of
radius $\delta>0$ centered at $\bar x$.
For a proper, closed, convex function $f$,
$\operatorname{dom}f$, $f^*$, and $\partial f$
denote its effective domain, Fenchel conjugate,
and convex subdifferential.

For a multifunction $\Gamma$, $\operatorname{dom}\Gamma$ and
$\operatorname{gph}\Gamma$ denote its domain and graph, with
$(\Gamma_2\circ\Gamma_1)(x):=\bigcup_{y\in\Gamma_1(x)}\Gamma_2(y)$.
For nonempty sets $C,D$, $\operatorname{dist}(x,C)$ and
$d_H(C,D)$ denote point-to-set and Hausdorff distances.
Further definitions appear in Section~\ref{preliminaries}.


\section{A composite representation theorem}
\label{sec2}

In this section, we establish a composite representation of the solution mapping in terms of the image mapping and the generalized basis pursuit mapping. We then derive complementary variational and geometric representations of these two constituent mappings.


\subsection{Main theorem}

Recall the solution mapping
\[
S(b,\lambda)
=
\arg\min_{x\in\mathbb R^n}
\left\{
F(Ax,b)+\lambda R(x)
\right\},
\qquad
(b,\lambda)\in\mathbb R^d\times\mathbb R_{++}.
\]
The following theorem gives its composite representation.

\begin{theorem}\label{T-decomposition}
Let $S$ be the solution mapping of problem~\eqref{P1}, and let the image mapping $H$ and the generalized basis pursuit mapping $\Phi$ be defined by \eqref{e-H-definition} and \eqref{e-Phi-definition}, respectively. Then, for every $(b,\lambda)\in\mathbb R^d\times\mathbb R_{++}$,
\[
S(b,\lambda)
=
(\Phi\circ H)(b,\lambda)
=
\bigcup_{y\in H(b,\lambda)}\Phi(y).
\]
\end{theorem}

\begin{proof}
We first show that
\begin{equation}\label{thm1-left}
\bigcup_{y\in H(b,\lambda)}\Phi(y)
\subseteq
S(b,\lambda).
\end{equation}
Let $y\in H(b,\lambda)$. By definition of $H$, there exists
$\bar x\in S(b,\lambda)$ such that $A\bar x=y$. For any
$x\in\Phi(y)$, we have $Ax=A\bar x=y$ and, by definition of $\Phi$,
$R(x)\leq R(\bar x)$. Hence
\[
F(Ax,b)+\lambda R(x)
		=
		F(y,b)+\lambda R(x)
		\le
		F(y,b)+\lambda R(\bar x)
		=
		F(A\bar x,b)+\lambda R(\bar x).
\]
Since $\bar x$ is a minimizer of problem~\eqref{P1}, equality must hold and $x$ is also a minimizer. Thus $x\in S(b,\lambda)$, proving \eqref{thm1-left}.

Conversely, take $x\in S(b,\lambda)$ and set $y:=Ax$. Then
$y\in H(b,\lambda)$. For any $x'\in\mathbb R^n$ satisfying $Ax'=y$, the optimality of $x$ gives
\[
F(y,b)+\lambda R(x)
		=
		F(Ax,b)+\lambda R(x)
		\le
		F(Ax',b)+\lambda R(x')
		=
		F(y,b)+\lambda R(x'). 
\]
Since $\lambda>0$, we obtain $R(x)\leq R(x')$, and hence
$x\in\Phi(y)$. Therefore,
\[
S(b,\lambda)
\subseteq
\bigcup_{y\in H(b,\lambda)}\Phi(y).
\]
Combining the two inclusions proves the result.
\end{proof}


	
	
\subsection{Representation of the image mapping}

The image mapping $H$ in \eqref{e-H-definition} is defined implicitly through the solution mapping $S$, which may make it difficult to analyze directly. We next characterize $H$ as the solution mapping of a reduced optimization problem in the image space. Recall from \eqref{e-V-definition} the value function
\[
V(y)
:=
\inf_{x\in\mathbb R^n}
\{R(x):Ax=y\},
\qquad
y\in\mathbb R^m,
\]
with the convention $V(y)=+\infty$ when $Ax=y$ is infeasible.

\begin{theorem}\label{H-representer}
Suppose that $R:\mathbb R^n\to(-\infty,+\infty]$ is proper, that $F(\cdot,b)$ is real-valued for every $b\in\mathbb R^d$, and that the infimum defining $V(y)$ is attained whenever $V(y)<+\infty$, i.e.,
\begin{equation}\label{e-fiber-attainment}
V(y)<+\infty
\quad\Longrightarrow\quad
\arg\min_{x\in\mathbb R^n}
\{R(x):Ax=y\}
\neq\emptyset.
\end{equation}
Then, for every
$(b,\lambda)\in\mathbb R^d\times\mathbb R_{++}$,
\begin{equation}\label{e-H-representation}
H(b,\lambda)
=
\arg\min_{y\in\mathbb R^m}
\left\{
F(y,b)+\lambda V(y)
\right\}.
\end{equation}
\end{theorem}

\begin{proof}
Fix $(b,\lambda)\in\mathbb R^d\times\mathbb R_{++}$ and set
\[
\vartheta
:=
\inf_{x\in\mathbb R^n}
\left\{
F(Ax,b)+\lambda R(x)
\right\}.
\]
By the definition of $V$ and grouping feasible points according to their image $y=Ax$,
\[
\vartheta
=
\inf_{y\in\mathbb R^m}
\left\{
F(y,b)+\lambda V(y)
\right\}.
\]

Let $y\in H(b,\lambda)$. Then there exists
$\bar x\in S(b,\lambda)$ such that $A\bar x=y$. Since
$V(y)\leq R(\bar x)$,
\[
F(y,b)+\lambda V(y)
\leq
F(A\bar x,b)+\lambda R(\bar x)
=
\vartheta.
\]
The reverse inequality follows from the definition of $\vartheta$ as the infimum of the reduced problem. Hence equality holds, and $y$ is a minimizer of the reduced problem.

Conversely, let $y$ minimize the right-hand side of \eqref{e-H-representation}. Its objective value is finite, so $V(y)<+\infty$. By \eqref{e-fiber-attainment}, there exists $\bar x$ such that
$A\bar x=y$ and $R(\bar x)=V(y)$. Therefore,
\[
F(A\bar x,b)+\lambda R(\bar x)
=
F(y,b)+\lambda V(y)
=
\vartheta,
\]
so $\bar x\in S(b,\lambda)$ and hence
$y=A\bar x\in H(b,\lambda)$. This proves \eqref{e-H-representation}.
\end{proof} 
	
\begin{remark}\label{R-H-representer-generality}
Theorem~\ref{H-representer} does not require convexity or continuity of $F(\cdot,b)$. The fiber-attainment condition~\eqref{e-fiber-attainment} is weaker than coercivity of $R$. It holds, for example, if $R$ is proper, lower semicontinuous, and level-bounded on every affine fiber of $A$, since the corresponding sublevel sets are then compact. For proper, closed, convex $R$, a sufficient condition for fiberwise level-boundedness is
\[
R^\infty(d)>0,
\qquad
\forall\,d\in\ker A\setminus\{0\},
\]
where $R^\infty$ denotes the recession function of $R$.

More precisely, Theorem~\ref{H-representer} only requires attainment of $V(y)$ at minimizers of the reduced problem; condition~\eqref{e-fiber-attainment} provides a parameter-independent sufficient condition. Without this assumption, one still has $H(b,\lambda)
\subseteq
\arg\min_y\{F(y,b)+\lambda V(y)\}$, but equality may fail if a reduced minimizer cannot be lifted to a primal minimizer.
\end{remark}

In the convex setting, we next derive a conjugate representation of $V$, which yields a corresponding characterization of $H$. 

\begin{corollary}\label{cor:H-representation}
Suppose that $F(\cdot,b)$ is real-valued for every $b\in\mathbb R^d$, and that $R:\mathbb R^n\to(-\infty,+\infty]$ is proper, closed, and convex. Assume further that
\begin{equation}\label{QC-H}
\operatorname{rge}(A^T)
\cap
\operatorname{ri}(\operatorname{dom}R^*)
\neq\emptyset,
\end{equation}
where
\[
R^*(v)
:=
\sup_{x\in\mathbb R^n}
\{\langle v,x\rangle-R(x)\}
\]
denotes the Fenchel conjugate of $R$. Then, for every
$(b,\lambda)\in\mathbb R^d\times\mathbb R_{++}$,
\[
H(b,\lambda)
=
\arg\min_{y\in\mathbb R^m}
\left\{
\frac{1}{\lambda}F(y,b)
+
(R^*\circ A^T)^*(y)
\right\}.
\]
\end{corollary}

\begin{proof}
Applying the exact conjugacy formula for linear compositions
\cite[Theorem~16.3]{Rockafellar1970} to $R^*$ and
$A^T:\mathbb R^m\to\mathbb R^n$, condition~\eqref{QC-H} gives
\[
(R^*\circ A^T)^*(y)
=
\min_{Ax=y}R^{**}(x)
=
\min_{Ax=y}R(x)
=
V(y)
\]
whenever the value is finite, where $R^{**}=R$ because $R$ is proper, closed, and convex. In particular, the infimum defining $V(y)$ is attained whenever $V(y)<+\infty$, so condition~\eqref{e-fiber-attainment} holds. The conclusion follows from Theorem~\ref{H-representer} by scaling by $\lambda^{-1}$. 
\end{proof} 
	
\begin{remark}\label{rem:H-representation}
Corollary~\ref{cor:H-representation} requires convexity only
of $R$, not of $F(\cdot,b)$. The qualification
condition~\eqref{QC-H} ensures both the exact conjugate
representation $V=(R^*\circ A^T)^*$ and fiberwise
attainment whenever $V(y)$ is finite. Thus,
condition~\eqref{e-fiber-attainment} in
Theorem~\ref{H-representer} holds automatically.

Unlike the generalized equation characterization of $H$
for convex $\ell_1$-regularized problems in
\cite{HuYaoZhang2025}, our representation uses a reduced
image-space problem involving the generalized basis pursuit
value function. The reduced representation in
Theorem~\ref{H-representer} applies more generally under
fiberwise attainment, without requiring convexity.
\end{remark}

We next use Corollary~\ref{cor:H-representation} to study
the local single-valuedness and Lipschitz continuity of $H$
via \cite[Theorem~4.4]{CuiHoheiselNghiaSun2026}.
We first recall the required $\mathcal C^2$-cone reducibility
property.

A closed convex set $\Theta\subseteq\mathcal E$ is
$\mathcal C^2$-cone reducible at $\bar v\in\Theta$ if
there exist a neighborhood $\mathcal U$ of $\bar v$,
a Euclidean space $W$, a pointed closed convex cone
$\mathcal K\subseteq W$, and a $\mathcal C^2$ mapping
$h:\mathcal U\to W$ such that $h(\bar v)=0$,
$Dh(\bar v)$ is surjective, and $\Theta\cap\mathcal U
=
\{v\in\mathcal U:h(v)\in\mathcal K\}$.  
A proper, closed, convex function $\varphi$ is
$\mathcal C^2$-cone reducible at
$\bar v\in\operatorname{dom}\varphi$ if
$\operatorname{epi}\varphi$ is $\mathcal C^2$-cone reducible
at $(\bar v,\varphi(\bar v))$.

\begin{theorem}\label{L-H-IFT}
Suppose that $R:\mathbb R^n\to(-\infty,+\infty]$ is proper, closed, and convex and satisfies condition~\eqref{QC-H}. Let
\[
\bar p=(\bar b,\bar\lambda),
\qquad
\bar y\in H(\bar p).
\]
Assume that $F$ is twice continuously differentiable in a neighborhood of $(\bar y,\bar b)$ and that $F(\cdot,b)$ is convex for every $b$ near $\bar b$. Define
\[
\bar z
:=
-\frac{1}{\bar\lambda}
\nabla_yF(\bar y,\bar b),
\qquad
\mathcal L
:=
\operatorname{par}(\partial R)^{-1}(A^T\bar z)
=
\operatorname{par}\partial R^*(A^T\bar z).
\]
Suppose further that $R^*\circ A^T$ is $\mathcal C^2$-cone reducible at $\bar z$ and
\begin{equation}\label{C-H-nondegeneracy}
\ker\nabla_{yy}^2F(\bar y,\bar b)
\cap
A\mathcal L
=
\{0\}.
\end{equation}
Then $H(\bar p)=\{\bar y\}$, and $H$ is single-valued and Lipschitz continuous in a neighborhood of $\bar p$.  
\end{theorem}

\begin{proof}
Set
\[
\psi:=(R^*\circ A^T)^*,
\qquad
\widehat F(y,p):=\frac{1}{\lambda}F(y,b),
\qquad
p=(b,\lambda).
\]
By Corollary~\ref{cor:H-representation},
\[
H(p)
=
\arg\min_{y\in\mathbb R^m}
\left\{
\widehat F(y,p)+\psi(y)
\right\}.
\]
Since $\bar\lambda>0$, $\widehat F$ is twice continuously differentiable near $(\bar y,\bar p)$, and $\widehat F(\cdot,p)$ is convex for every $p$ sufficiently close to $\bar p$. We now verify the remaining conditions of \cite[Theorem~4.4]{CuiHoheiselNghiaSun2026}.

Since $\psi^*=R^*\circ A^T$, condition~\eqref{QC-H} and the subdifferential chain rule give
\[
\partial\psi^*(z)
=
A\partial R^*(A^Tz)
=
A(\partial R)^{-1}(A^Tz).
\]
Hence, using the nonempty convexity of
$(\partial R)^{-1}(A^T\bar z)$,
\[
\operatorname{par}\partial\psi^*(\bar z)
=
A\operatorname{par}(\partial R)^{-1}(A^T\bar z)
=
A\mathcal L.
\]
The first-order optimality condition at $\bar y\in H(\bar p)$ yields
\[
\bar z
=
-\nabla_y\widehat F(\bar y,\bar p)
=
-\frac{1}{\bar\lambda}\nabla_yF(\bar y,\bar b)
\in
\partial\psi(\bar y).
\]
Moreover, $\nabla_{yy}^2\widehat F(\bar y,\bar p)
=
\frac{1}{\bar\lambda}
\nabla_{yy}^2F(\bar y,\bar b)$, 
and hence $\ker\nabla_{yy}^2\widehat F(\bar y,\bar p)
=
\ker\nabla_{yy}^2F(\bar y,\bar b)$. 
Therefore, condition~\eqref{C-H-nondegeneracy} is equivalent to
\[
\ker\nabla_{yy}^2\widehat F(\bar y,\bar p)
\cap
\operatorname{par}\partial\psi^*(\bar z)
=
\{0\}.
\]

Finally, $\psi^*=R^*\circ A^T$ is $\mathcal C^2$-cone reducible at $\bar z$ by assumption. Thus all the conditions of \cite[Theorem~4.4]{CuiHoheiselNghiaSun2026} are satisfied. It follows that
$H(\bar p)=\{\bar y\}$ and that $H$ is single-valued and Lipschitz continuous in a neighborhood of $\bar p$.  
\end{proof} 

\begin{remark}\label{remark_stability_H}
The image-space formulation can yield a less restrictive
sufficient condition than direct primal-space analysis using
\cite[Theorem~4.4]{CuiHoheiselNghiaSun2026}.
For example, if $\nabla_{yy}^2F(\bar y,\bar b)\succ0$,
as for the squared loss, condition~\eqref{C-H-nondegeneracy}
holds automatically, whereas the primal-space Hessian
$A^T\nabla_{yy}^2F(\bar y,\bar b)A$ may be singular.
Section~\ref{image-primal} shows that the primal-space
sufficient conditions imply the image-space conditions,
but not conversely in general. 
\end{remark} 

	
\subsection{A geometric representation of the generalized basis pursuit mapping}

Having obtained representations of the image mapping $H$, we now turn to the generalized basis pursuit mapping
\[
\Phi(y)
=
\arg\min_{x\in\mathbb R^n}
\left\{
R(x):Ax=y
\right\}.
\]
We derive a geometric representation of $\Phi$ as the intersection of an affine fiber with a fixed set determined by $(A,R)$. 

Let $K:=\ker A$. For $y\in\operatorname{rge}A$, define the affine fiber $A^{-1}(y)
:=
\{x\in\mathbb R^n:Ax=y\}$. 
Since $A^{-1}(y)$ is an affine subspace with direction space $K$, its normal cone is
\[
N_{A^{-1}(y)}(x)
=
K^\perp
=
\operatorname{rge}(A^T),
\qquad
x\in A^{-1}(y).
\]
Define the fixed lifting set associated with $A$ and $R$ by
\begin{equation}\label{C-def}
C_R
:=
(\partial R)^{-1}(K^\perp)
=
\left\{
x\in\mathbb R^n:
\partial R(x)\cap K^\perp\neq\emptyset
\right\}.
\end{equation}
Thus, $C_R$ consists of points at which $R$ admits a subgradient orthogonal to the directions of the affine fibers. In particular, $C_R$ depends only on $(A,R)$ and not on the image variable $y$.

\begin{proposition}[Geometric representation of $\Phi$]
\label{P-Phi-geometry}
Suppose that $R:\mathbb R^n\to(-\infty,+\infty]$ is proper, closed, and convex, and that
\begin{equation}\label{QC-Phi}
A^{-1}(y)\cap\operatorname{ri}(\operatorname{dom}R)
\neq\emptyset,
\qquad
\forall\,y\in\operatorname{rge}A.
\end{equation}
Then
\begin{equation}\label{Phi-C}
\Phi(y)
=
A^{-1}(y)\cap C_R,
\qquad
y\in\operatorname{rge}A.
\end{equation}
Moreover, $\operatorname{dom}\Phi=A(C_R)$.  
\end{proposition}

\begin{proof}
Fix $y\in\operatorname{rge}A$. By \eqref{QC-Phi}, the relative-interior qualification for $R$ and $\delta_{A^{-1}(y)}$ holds. Hence, for every $x\in A^{-1}(y)$, the convex subdifferential sum rule gives
\[
\partial\bigl(R+\delta_{A^{-1}(y)}\bigr)(x)
=
\partial R(x)+N_{A^{-1}(y)}(x)
=
\partial R(x)+K^\perp.
\]
Therefore,
\[
x\in\Phi(y)
\Longleftrightarrow
0\in\partial\bigl(R+\delta_{A^{-1}(y)}\bigr)(x)
\Longleftrightarrow
0\in\partial R(x)+K^\perp
\Longleftrightarrow
\partial R(x)\cap K^\perp\neq\emptyset, 
\] 
where the last equivalence uses the fact that $K^\perp$ is a linear subspace. Thus,
\[
x\in\Phi(y)
\Longleftrightarrow
x\in A^{-1}(y)\cap C_R,
\]
which proves \eqref{Phi-C}. The domain characterization follows immediately:
\[
y\in\operatorname{dom}\Phi
\Longleftrightarrow
A^{-1}(y)\cap C_R\neq\emptyset
\Longleftrightarrow
y\in A(C_R).
\] 
\end{proof}

Condition~\eqref{QC-Phi} holds automatically when $R$ is
finite-valued on $\mathbb R^n$, as for the $\ell_1$ and
$\ell_2$ norms. In these cases, $C_R$ can often be
characterized explicitly.

For $R(x)=\|x\|^2/2$, we have
$\partial R(x)=\{x\}$ and hence
$C_R=K^\perp=\operatorname{rge}(A^T)$.
Since every affine fiber intersects $K^\perp$ at exactly
one point, its unique minimum-norm element,
\[
\Phi(y)=\{A^\dagger y\},
\qquad y\in\operatorname{rge}A.
\]

For $R(x)=\|x\|_1$, $x\in C_R$ if and only if
there exists $z\in\mathbb R^m$ such that
\[
(A^Tz)_i=\operatorname{sgn}(x_i)\ \text{if }x_i\neq0,
\qquad
|(A^Tz)_i|\leq1\ \text{otherwise}.
\]
Thus $C_R$ is characterized by dual certificates
and is a finite union of convex polyhedral cones,
not necessarily convex. 

Proposition~\ref{P-Phi-geometry} separates the dependence
on $y$, through the affine fiber $A^{-1}(y)$, from the
fixed geometry of $C_R$. To examine the variation of
these fibers, let
$B:=A|_{K^\perp}:K^\perp\to\operatorname{rge}A$.
Since $B$ is a linear isomorphism,
$A^{-1}(y)=B^{-1}(y)+K$ for $y\in\operatorname{rge}A$.
Consequently,
\[
d_H\bigl(A^{-1}(y_1),A^{-1}(y_2)\bigr)
=
\|B^{-1}(y_1)-B^{-1}(y_2)\|
\leq
\|B^{-1}\|\,\|y_1-y_2\|.
\]
Thus the affine fibers vary Lipschitz continuously,
and the stability of $\Phi$ depends on their intersections
with the fixed lifting set $C_R$. 
	
	
\section{Applications to Lipschitz stability}\label{sec3}

We apply the composite representation $S=\Phi\circ H$
to study the local Hausdorff Lipschitz stability of $S$.
We first establish a composition rule for set-valued mappings,
then compare image-space and primal-space stability conditions,
and finally derive stability results for representative
classes of regularizers.


\subsection{Preliminaries}
\label{preliminaries}

We recall several notions used in the stability analysis.

For a nonempty convex set $C$, its parallel subspace is
$\operatorname{par}C:=\operatorname{span}(C-x)$ for any
$x\in C$, and its polar cone is
$C^\circ:=\{u:\langle u,v\rangle\leq0,\ \forall v\in C\}$.
A convex set is polyhedral if it is the intersection of
finitely many closed halfspaces.

For nonempty sets $C,D$, their excess and Hausdorff
distance are
\[
e(C,D):=\sup_{x\in C}\operatorname{dist}(x,D),
\qquad
d_H(C,D):=\max\{e(C,D),e(D,C)\},
\]
where $\operatorname{dist}(x,D):=\inf_{y\in D}\|x-y\|$.
These quantities may be infinite.

Let $\Gamma:\mathbb R^m\rightrightarrows\mathbb R^n$
and $D\subseteq\operatorname{dom}\Gamma$.
The mapping $\Gamma$ is Hausdorff Lipschitz continuous
relative to $D$ if there exists $L>0$ such that
\[
d_H\bigl(\Gamma(x_1),\Gamma(x_2)\bigr)
\leq L\|x_1-x_2\|,
\qquad x_1,x_2\in D.
\]
It is locally Hausdorff Lipschitz continuous at
$\bar x\in D$ relative to $D$ if this estimate holds
for all $x_1,x_2\in D\cap B(\bar x,\delta)$
for some $\delta,L>0$.

The mapping $\Gamma$ is outer Lipschitz continuous at
$\bar x\in D$ relative to $D$ if $\Gamma(\bar x)$ is
closed and there exist $\delta,\kappa>0$ such that
\[
e\bigl(\Gamma(x),\Gamma(\bar x)\bigr)
\leq\kappa\|x-\bar x\|,
\qquad x\in D\cap B(\bar x,\delta).
\]
It is outer Lipschitz continuous relative to $D$
with modulus $\kappa$ if this estimate holds at
every $\bar x\in D$ with the same $\kappa$,
allowing $\delta$ to depend on $\bar x$.

For $\{x^k\}\subseteq D$, write
$x^k\xrightarrow{D}\bar x$ if $x^k\to\bar x$.
The inner limit of $\Gamma$ at $\bar x\in D$
relative to $D$ is
\[
\liminf_{x\xrightarrow{D}\bar x}\Gamma(x)
:=
\left\{
y:
\begin{array}{l}
\text{for every }x^k\xrightarrow{D}\bar x,
\text{ there exist }y^k\in\Gamma(x^k)\\
\text{for all sufficiently large }k
\text{ such that }y^k\to y
\end{array}
\right\}.
\]
The mapping $\Gamma$ is inner semicontinuous at
$\bar x\in D$ relative to $D$ if $\Gamma(\bar x)
\subseteq
\liminf_{x\xrightarrow{D}\bar x}\Gamma(x)$. 
It is inner semicontinuous relative to $D$ if
this holds at every $\bar x\in D$. 

	
\subsection{A composition rule for Hausdorff Lipschitz stability}

We begin with a composition rule for Hausdorff Lipschitz continuity of set-valued mappings, which will later be applied to $S=\Phi\circ H$.

\begin{proposition}\label{P-transfer}
Let $G_1:\mathbb R^r\rightrightarrows\mathbb R^m$ and
$G_2:\mathbb R^m\rightrightarrows\mathbb R^n$ be set-valued mappings, and define
$G:=G_2\circ G_1$. Fix $\bar p\in\operatorname{dom}G$. Suppose that there exist a neighborhood $U$ of $\bar p$ and a set
$V\subseteq\operatorname{dom}G_2$ such that
\[
G_1(p)\subseteq V,
\qquad
\forall\,p\in U\cap\operatorname{dom}G.
\]
Assume that $G_1$ is Hausdorff Lipschitz continuous on
$U\cap\operatorname{dom}G$ with modulus $L_1$, and that $G_2$ is Hausdorff Lipschitz continuous on $V$ with modulus $L_2$. Then $G$ is Hausdorff Lipschitz continuous on $U\cap\operatorname{dom}G$ with modulus $L_2L_1$. In particular, $G$ is locally Hausdorff Lipschitz continuous at $\bar p$ relative to $\operatorname{dom}G$, with
\[
d_H\bigl(G(p_1),G(p_2)\bigr)
\leq
L_2L_1\|p_1-p_2\|,
\qquad
\forall\,p_1,p_2\in U\cap\operatorname{dom}G.
\]
\end{proposition}

\begin{proof}
Fix $p_i\in U\cap\operatorname{dom}G$ for $i=1,2$, and set
$Y_i:=G_1(p_i)$ and $X_i:=G(p_i)$. Then
\[
X_i
=
\bigcup_{y\in Y_i}G_2(y),
\qquad
i=1,2.
\]

We first estimate $e(X_1,X_2)$. Fix $x_1\in X_1$ and choose
$y_1\in Y_1$ such that $x_1\in G_2(y_1)$. For any $\eta>0$, there exists
$y_2\in Y_2$ such that
\begin{equation}\label{dist_y}
\|y_1-y_2\|
\leq
\operatorname{dist}(y_1,Y_2)+\eta
\leq
d_H(Y_1,Y_2)+\eta.
\end{equation}
Since $Y_i\subseteq V$, the Hausdorff Lipschitz continuity of $G_2$ gives
\[
d_H\bigl(G_2(y_1),G_2(y_2)\bigr)
\leq
L_2\|y_1-y_2\|.
\]
Choose $x_2\in G_2(y_2)\subseteq X_2$ such that
\[
\|x_1-x_2\|
\leq
\operatorname{dist}\bigl(x_1,G_2(y_2)\bigr)+\eta
\leq
d_H\bigl(G_2(y_1),G_2(y_2)\bigr)+\eta.
\]
Combining these estimates with \eqref{dist_y} yields
\[
\operatorname{dist}(x_1,X_2)
\leq
L_2d_H(Y_1,Y_2)+(L_2+1)\eta.
\]
Letting $\eta\downarrow0$ and taking the supremum over $x_1\in X_1$, we obtain $e(X_1,X_2)
\leq
L_2d_H(Y_1,Y_2)$.  
By symmetry,
$e(X_2,X_1)\leq L_2d_H(Y_1,Y_2)$, and therefore
\[
d_H(X_1,X_2)
\leq
L_2d_H(Y_1,Y_2)
\leq
L_2L_1\|p_1-p_2\|.
\]
\end{proof}

\begin{remark}
Proposition~\ref{P-transfer} complements standard composition rules
for the Aubin property of set-valued mappings
\cite[Proposition~3F.10]{DontchevRockafellar2009}.
Here, Hausdorff Lipschitz continuity of the composition
follows directly from the corresponding estimates for its components.
\end{remark}

Applying Proposition~\ref{P-transfer} to $S=\Phi\circ H$
allows the stability estimates for $H$ and $\Phi$ to be
combined. We next compare the image-space and primal-space
stability conditions. 

	
\subsection{Image-space versus primal-space stability conditions}
\label{image-primal}

Theorem~\ref{L-H-IFT} establishes sufficient conditions for
local single-valuedness and Lipschitz continuity of $H$.
We compare these conditions with those obtained by applying
\cite[Theorem~4.4]{CuiHoheiselNghiaSun2026} directly to $S$.

Since $H=A\circ S$ and $A$ is linear,
\[
d_H\bigl(H(p_1),H(p_2)\bigr)
\leq
\|A\|\,d_H\bigl(S(p_1),S(p_2)\bigr).
\]
Thus, Hausdorff Lipschitz continuity of $S$ implies that
of $H$, and local single-valued Lipschitz continuity
of $S$ implies the corresponding property of $H$.
The converse need not hold.

\begin{example}\label{E-H-single-S-set}
Let $A=(1,1)$, $R(x)=\|x\|_1$, and
$F(y,b)=(y-b)^2/2$. Since
$V(y)=\min_{Ax=y}\|x\|_1=|y|$, $H(b,\lambda)=\{\operatorname{sgn}(b)(|b|-\lambda)_+\}$ is single-valued and Lipschitz continuous.
However, for $b>\lambda>0$, 
\[
S(b,\lambda)
=
\Phi(b-\lambda)
=
\{x\in\mathbb R_+^2:x_1+x_2=b-\lambda\},
\]
which is a nontrivial line segment.
Thus, $H$ can be single-valued and Lipschitz continuous
even when $S$ is set-valued.
\end{example}

We now compare the sufficient conditions in the two spaces.
Fix $\bar p=(\bar b,\bar\lambda)\in\operatorname{dom}S$
and $\bar x\in S(\bar p)$, and set $\bar y:=A\bar x$.
Suppose that $R$ is proper, closed, and convex, that $F$
is twice continuously differentiable near $(\bar y,\bar b)$,
and that $F(\cdot,b)$ is convex for every $b$ near $\bar b$.
Define
\[
\bar z:=-\bar\lambda^{-1}\nabla_yF(\bar y,\bar b),
\qquad
Q:=\nabla_{yy}^2F(\bar y,\bar b),
\qquad
\mathcal L:=\operatorname{par}\partial R^*(A^T\bar z).
\]
Then $Q\succeq0$, and optimality gives $A^T\bar z\in\partial R(\bar x)$, equivalently, $\bar x\in\partial R^*(A^T\bar z)$. 

Applying \cite[Theorem~4.4]{CuiHoheiselNghiaSun2026}
directly to $S$ yields the following sufficient conditions
for local single-valued Lipschitz continuity.

\begin{condition}[Primal-space sufficient conditions]
\label{C-primal-stability}
\begin{align}
& R^*
\text{ is $\mathcal C^2$-cone reducible at }A^T\bar z,
\label{e-primal-structural}\\
& \ker(A^TQA)\cap\mathcal L=\{0\}.
\label{e-primal-ND}
\end{align}
\end{condition}

Applying the same result to $H$, as in
Theorem~\ref{L-H-IFT}, gives the corresponding
image-space sufficient conditions.

\begin{condition}[Image-space sufficient conditions]
\label{C-image-stability}
Condition~\eqref{QC-H} holds, and
\begin{align}
& R^*\circ A^T
\text{ is $\mathcal C^2$-cone reducible at }\bar z,
\label{e-image-structural}\\
& \ker Q\cap A\mathcal L=\{0\}.
\label{e-image-ND-comparison}
\end{align}
\end{condition}

We first compare the two nondegeneracy conditions. 

\begin{theorem}\label{T-primal-image-ND}
Let $A\in\mathbb R^{m\times n}$, $Q\in\mathbb R^{m\times m}$ be symmetric positive semidefinite,
and $\mathcal L\subseteq\mathbb R^n$ be a subspace. Then \eqref{e-primal-ND} holds if and only if \eqref{e-image-ND-comparison} holds and 
\begin{equation}\label{e-A-injectivity}
\ker A\cap\mathcal L=\{0\}.
\end{equation}
\end{theorem}

\begin{proof}
Since $Q\succeq0$,
\[
\langle A^TQAd,d\rangle
=
\|Q^{1/2}Ad\|^2,
\qquad
d\in\mathbb R^n,
\]
and hence
\[
\ker(A^TQA)
=
\{d\in\mathbb R^n:Ad\in\ker Q\}.
\]

Suppose first that
$\ker(A^TQA)\cap\mathcal L=\{0\}$.
If $d\in\ker A\cap\mathcal L$, then
$Ad=0\in\ker Q$, so the kernel identity gives $d=0$. Thus
\eqref{e-A-injectivity} holds. If
$\eta\in\ker Q\cap A\mathcal L$, write $\eta=Ad$ for some
$d\in\mathcal L$. Then
$d\in\ker(A^TQA)\cap\mathcal L$, so $d=0$ and hence $\eta=0$.
Thus \eqref{e-image-ND-comparison} holds.

Conversely, suppose that
$\ker Q\cap A\mathcal L=\{0\}$ and
$\ker A\cap\mathcal L=\{0\}$. If
$d\in\ker(A^TQA)\cap\mathcal L$, then
$Ad\in\ker Q\cap A\mathcal L$, so $Ad=0$. Hence
$d\in\ker A\cap\mathcal L=\{0\}$, proving \eqref{e-primal-ND}.
\end{proof}

The injectivity condition~\eqref{e-A-injectivity} also relates
the qualification and $\mathcal C^2$-cone reducibility
conditions in the primal and image spaces.

\begin{proposition}\label{P-primal-image-structural}
Let $R:\mathbb R^n\to(-\infty,+\infty]$ be proper,
closed, and convex, and suppose that
$A^T\bar z\in\partial R(\bar x)$ for some
$\bar x\in\mathbb R^n$ and $\bar z\in\mathbb R^m$.
Set $\mathcal L:=\operatorname{par}\partial R^*(A^T\bar z)$.
If \eqref{e-A-injectivity} holds, then \eqref{QC-H} holds.
If, in addition, \eqref{e-primal-structural} holds,
then \eqref{e-image-structural} also holds.
\end{proposition}

\begin{proof}
Since $A^T\bar z\in\partial R(\bar x)$,
we have $\bar x\in\partial R^*(A^T\bar z)$.
Moreover,
\begin{equation}\label{e-normal-L-inclusion}
N_{\operatorname{dom}R^*}(A^T\bar z)
\subseteq
\operatorname{par}\partial R^*(A^T\bar z)
=
\mathcal L.
\end{equation}
Indeed, for any
$d\in N_{\operatorname{dom}R^*}(A^T\bar z)$,
the subgradient and normal-cone inequalities imply
$\bar x+d\in\partial R^*(A^T\bar z)$.

Suppose that \eqref{QC-H} fails.
Since
$A^T\bar z\in\operatorname{rge}(A^T)\cap
\operatorname{dom}R^*$,
the proper separation theorem
\cite[Theorem~11.3]{Rockafellar1970} yields
\[
0\neq w\in
\operatorname{rge}(A^T)^\perp
\cap N_{\operatorname{dom}R^*}(A^T\bar z).
\]
By \eqref{e-normal-L-inclusion} and
$\operatorname{rge}(A^T)^\perp=\ker A$, $0\neq w\in\ker A\cap\mathcal L$,
contradicting \eqref{e-A-injectivity}.
Hence \eqref{QC-H} holds.

Assume also \eqref{e-primal-structural}, and set $\bar t:=R^*(A^T\bar z)$ and $\bar v:=(A^T\bar z,\bar t)$.
Then $\mathcal C^2$-cone reducibility yields a neighborhood $\mathcal U$ of $\bar v$, a pointed closed convex cone $\mathcal K$ in a Euclidean space $W$,
and a $\mathcal C^2$ map $h:\mathcal U\to W$ with $h(\bar v)=0$ and surjective $Dh(\bar v)$ such that 
\begin{equation}\label{e-reduction-Rstar}
\operatorname{epi}R^*\cap\mathcal U
=
\{(u,t)\in\mathcal U:h(u,t)\in\mathcal K\}.
\end{equation}

Define
$\mathcal A(z,t):=(A^Tz,t)$ and
$\widehat{\mathcal U}:=\mathcal A^{-1}(\mathcal U)$.
Since
$\operatorname{epi}(R^*\circ A^T)
=\mathcal A^{-1}(\operatorname{epi}R^*)$,
we obtain
\begin{equation}\label{e-reduction-composition}
\operatorname{epi}(R^*\circ A^T)
\cap\widehat{\mathcal U}
=
\{(z,t)\in\widehat{\mathcal U}:
(h\circ\mathcal A)(z,t)\in\mathcal K\}.
\end{equation}
It remains to verify that
$D(h\circ\mathcal A)(\bar z,\bar t)$ is surjective.

By \cite[Theorem~6.14]{RW1998},
\[
N_{\operatorname{epi}R^*}(\bar v)
=
Dh(\bar v)^*N_{\mathcal K}(0).
\]
Since $\mathcal K$ is pointed,
$N_{\mathcal K}(0)=\mathcal K^\circ$
is full-dimensional. Therefore,
\begin{equation}\label{e-range-reduction-adjoint}
\operatorname{rge}Dh(\bar v)^*
=
\operatorname{span}N_{\operatorname{epi}R^*}(\bar v).
\end{equation}
Furthermore,
\cite[Proposition~2.1]{CuiHoheiselNghiaSun2026}
gives
\[
\operatorname{span}N_{\operatorname{epi}R^*}(\bar v)
\cap(\mathbb R^n\times\{0\})
=
\mathcal L\times\{0\}.
\]
Since
$\mathcal A^*(d,\beta)=(Ad,\beta)$,
we have $\ker\mathcal A^*=\ker A\times\{0\}$.
Consequently,
\[
\ker\mathcal A^*
\cap\operatorname{rge}Dh(\bar v)^*
=
(\ker A\cap\mathcal L)\times\{0\}
=
\{0\},
\]
where the last equality follows from
\eqref{e-A-injectivity}.

As $Dh(\bar v)$ is surjective, its adjoint is
injective. The preceding identity therefore implies
that $\mathcal A^*Dh(\bar v)^*$ is injective.
Hence
$D(h\circ\mathcal A)(\bar z,\bar t)
=Dh(\bar v)\circ\mathcal A$
is surjective. Thus \eqref{e-reduction-composition} defines a
$\mathcal C^2$-cone reduction of
$\operatorname{epi}(R^*\circ A^T)$ at
$(\bar z,\bar t)$, proving \eqref{e-image-structural}.
\end{proof} 
    
Combining Theorem~\ref{T-primal-image-ND} and
Proposition~\ref{P-primal-image-structural} gives the
following relation between the two sets of sufficient conditions.

\begin{corollary}\label{T-primal-image-hierarchy}
Let $\bar p=(\bar b,\bar\lambda)\in\operatorname{dom}S$,
$\bar x\in S(\bar p)$, and $\bar y:=A\bar x$.
Suppose that $R$ is proper, closed, and convex,
that $F$ is twice continuously differentiable near
$(\bar y,\bar b)$, and that $F(\cdot,b)$ is convex
for every $b$ near $\bar b$.
Then Condition~\ref{C-primal-stability} implies
Condition~\ref{C-image-stability}.
\end{corollary}

\begin{proof}
By Theorem~\ref{T-primal-image-ND},
\eqref{e-primal-ND} implies
\eqref{e-image-ND-comparison} and
\eqref{e-A-injectivity}.
Proposition~\ref{P-primal-image-structural},
together with \eqref{e-primal-structural},
then yields \eqref{QC-H} and
\eqref{e-image-structural}.
\end{proof}

Thus, the image-space sufficient conditions are weaker
than their primal-space counterparts in the sense that
the latter imply the former.

\begin{remark}
Condition~\eqref{e-A-injectivity} excludes nonzero
directions in $\mathcal L$ that lie in $\ker A$.
In Example~\ref{E-H-single-S-set},
$\mathcal L=\mathbb R^2$ and
$\ker A=\operatorname{span}\{(1,-1)\}$,
so this condition fails.
The kernel direction $(1,-1)$ allows distinct primal
solutions to share the same image, explaining why
$H$ is single-valued and Lipschitz continuous
while $S$ remains set-valued.
\end{remark}

\begin{remark}
Condition~\eqref{e-A-injectivity} also implies
$\Phi(\bar y)=\{\bar x\}$.
Indeed, for any $x\in\Phi(\bar y)$,
we have $Ax=A\bar x$ and $R(x)=R(\bar x)$.
Since $A^T\bar z\in\partial R(\bar x)$,
equality in the subgradient inequality gives
$A^T\bar z\in\partial R(x)$.
Hence
$x,\bar x\in\partial R^*(A^T\bar z)$,
so $x-\bar x\in\mathcal L$.
Together with $x-\bar x\in\ker A$,
condition~\eqref{e-A-injectivity} yields
$x=\bar x$. However, \eqref{e-A-injectivity} is not required
for the Hausdorff Lipschitz analysis below,
which allows $\Phi$ to remain set-valued.
\end{remark} 

	
\subsection{Polyhedral geometry and Lipschitz stability of the generalized basis pursuit mapping}

Throughout this subsection, suppose that $R$ is proper,
closed, and convex, and that
\[
A^{-1}(y)\cap\operatorname{ri}(\operatorname{dom}R)
\neq\emptyset,
\qquad
\forall\,y\in\operatorname{rge}A.
\]
By Proposition~\ref{P-Phi-geometry},
$\Phi(y)=A^{-1}(y)\cap C_R$ for
$y\in\operatorname{rge}A$.
We study the stability of $\Phi$ through these affine
slices, first when $C_R$ is convex polyhedral and then
when it is a finite union of convex polyhedral sets.

\begin{proposition}\label{P-Phi-poly}
Suppose that $C_R$ is convex polyhedral.
Then $\Phi$ is Hausdorff Lipschitz continuous relative
to $\operatorname{dom}\Phi$. Specifically, there exists
$L_\Phi>0$ such that
\[
d_H\bigl(\Phi(y_1),\Phi(y_2)\bigr)
\leq L_\Phi\|y_1-y_2\|,
\qquad
\forall\,y_1,y_2\in\operatorname{dom}\Phi.
\]
\end{proposition}

\begin{proof}
By Proposition~\ref{P-Phi-geometry},
\[
\operatorname{gph}\Phi
=
\{(y,x):Ax=y,\ x\in C_R\}.
\]
Since $C_R$ is convex polyhedral, so is $\operatorname{gph}\Phi$, and the result follows from \cite[Example~9.35]{RW1998}. 
\end{proof} 

We next consider finite unions of convex polyhedral sets.
Although each branch is Hausdorff Lipschitz continuous relative
to its domain, the union need not be, since a branch may become
infeasible under arbitrarily small perturbations. Suppose that
\begin{equation}\label{CR-poly-union}
C_R=\bigcup_{\ell=1}^N C_\ell,
\end{equation}
where each $C_\ell$ is convex polyhedral. Define
$\Phi_\ell(y):=A^{-1}(y)\cap C_\ell$. Then
\[
\Phi(y)=\bigcup_{\ell=1}^N\Phi_\ell(y),
\qquad
\operatorname{dom}\Phi_\ell=A(C_\ell).
\]
Each $\Phi_\ell$ has a convex polyhedral graph and is
Hausdorff Lipschitz continuous relative to its domain
by the same argument as in Proposition~\ref{P-Phi-poly}.
Moreover, $\Phi$ is outer Lipschitz continuous relative
to $\operatorname{dom}\Phi$ with a uniform modulus,
since its graph is a finite union of convex polyhedral sets;
see \cite[Theorem~3D.1]{DontchevRockafellar2009}.

For closed-valued multifunctions on convex parameter sets,
uniform outer Lipschitz continuity and inner semicontinuity
together imply Hausdorff Lipschitz continuity
\cite[Theorem~3D.3]{DontchevRockafellar2009}.
To ensure inner semicontinuity, we introduce the following condition.

\begin{condition}[Branch persistence]
Let $D\subseteq\operatorname{dom}\Phi$ be convex.
Branch persistence holds on $D$ with respect to
\eqref{CR-poly-union} if, for every
$\bar y\in D$ and $\bar x\in\Phi(\bar y)$,
there exist $\ell\in\{1,\ldots,N\}$ and $\delta>0$ such that
\begin{equation}\label{e-branch-persistence}
\bar x\in\Phi_\ell(\bar y),
\qquad
D\cap B(\bar y,\delta)
\subseteq\operatorname{dom}\Phi_\ell.
\end{equation}
\end{condition}

Thus, every reference solution belongs to at least one
branch that remains feasible for all nearby image values in $D$.

\begin{proposition}\label{P-union-poly-HLip}
Suppose that $C_R$ admits a finite convex-polyhedral
representation~\eqref{CR-poly-union}.
Let $D\subseteq\operatorname{dom}\Phi$ be convex.
If branch persistence~\eqref{e-branch-persistence}
holds on $D$, then $\Phi$ is Hausdorff Lipschitz
continuous relative to $D$.
\end{proposition}

\begin{proof}
By the preceding discussion, $\Phi$ is uniformly outer
Lipschitz continuous relative to $D$ and is closed-valued.
It remains to establish inner semicontinuity.

Fix $\bar y\in D$, $\bar x\in\Phi(\bar y)$, and
$y^k\xrightarrow{D}\bar y$.
By branch persistence, there exists $\ell$ such that
$\bar x\in\Phi_\ell(\bar y)$ and
$y^k\in\operatorname{dom}\Phi_\ell$ for all sufficiently
large $k$. Since $\Phi_\ell$ is Hausdorff Lipschitz
continuous relative to its domain,
\[
\operatorname{dist}\bigl(\bar x,\Phi_\ell(y^k)\bigr)
\leq
d_H\bigl(\Phi_\ell(\bar y),\Phi_\ell(y^k)\bigr)
\longrightarrow0.
\]
Hence some $x^k\in\Phi_\ell(y^k)\subseteq\Phi(y^k)$ satisfy
$x^k\to\bar x$, proving inner semicontinuity relative to $D$. 

Since $D$ is convex, the conclusion follows from
\cite[Theorem~3D.3]{DontchevRockafellar2009}.
\end{proof}

Examples in Section~\ref{sec:branch-persistence-examples} of the Appendix
illustrate that branch disappearance can destroy Hausdorff
continuity, while branch persistence is sufficient but
not necessary for Hausdorff Lipschitz continuity. 

	
\subsection{Illustrative examples}\label{Representative regularizers}

We consider two classes of regularizers: structured piecewise
linear-quadratic regularizers, for which Hausdorff Lipschitz
stability follows under branch persistence, and $\ell_1$-type
regularizers, for which stability can be established without
this condition.


\subsubsection{Structured PLQ regularizers}

General convex piecewise linear-quadratic (PLQ) regularizers
need not satisfy the $\mathcal C^2$-cone reducibility condition
in Theorem~\ref{L-H-IFT}; see
\cite[Example~5.9]{SmulanskyHoheiselNghia2026}.
We therefore impose additional structure to ensure this property.

Let
$P=\{v\in\mathbb R^n:\langle p_i,v\rangle\leq\beta_i,\
i=1,\ldots,N\}$
be a nonempty convex polyhedral set,
$\Sigma\in\mathbb R^{n\times n}$ be symmetric positive
semidefinite, and $L:\mathbb R^n\to\mathbb R^n$
be a linear isomorphism. Following
\cite[Definition~5.10]{SmulanskyHoheiselNghia2026},
define
\begin{equation}\label{e-PLQ-RP}
R_P(x)
:=
\theta_{P,\Sigma}(Lx)
:=
\sup_{v\in P}
\left\{
\langle Lx,v\rangle
-\frac12\langle\Sigma v,v\rangle
\right\}.
\end{equation}
Since $R_P$ is proper, closed, convex, and PLQ \cite[Example~11.18]{RW1998},
$\partial R_P$ is piecewise polyhedral, with $\operatorname{gph}\partial R_P$ a finite union of convex polyhedral sets $G_1,\ldots,G_M\subseteq\mathbb R^n\times\mathbb R^n$: 
\begin{equation}\label{e-PLQ-subdiff-decomposition}
\operatorname{gph}\partial R_P
=
\bigcup_{\ell=1}^M G_\ell.
\end{equation}

Recall that $K=\ker A$ and
$K^\perp=\operatorname{rge}(A^T)$. Define
\begin{equation}\label{e-PLQ-branch}
C_\ell
:=
\operatorname{proj}_x
\left(
G_\ell\cap(\mathbb R^n\times K^\perp)
\right),
\qquad \ell=1,\ldots,M.
\end{equation}
Each $C_\ell$ is convex polyhedral, and
\begin{equation}\label{e-PLQ-CR-decomposition}
C_{R_P}
=
(\partial R_P)^{-1}(K^\perp)
=
\bigcup_{\ell=1}^M C_\ell.
\end{equation}

Finally, let $P^\infty$ denote the recession cone of $P$.
The domain of $R_P$ is
\begin{equation}\label{e-PLQ-domain}
\operatorname{dom}R_P
=
L^{-1}\left((P^\infty\cap\ker\Sigma)^\circ\right).
\end{equation}
Indeed, this follows from $\operatorname{dom}(\frac12\langle\Sigma\cdot,\cdot\rangle+\delta_P)^*
=\operatorname{rge}\Sigma+(P^\infty)^\circ=(P^\infty\cap\ker\Sigma)^\circ$.

\begin{proposition}\label{T-structured-PLQ}
Let $R_P$ be defined by \eqref{e-PLQ-RP}, and let
$S_P$, $H_P$, and $\Phi_P$ denote the corresponding
solution, image, and generalized basis pursuit mappings.
Fix $\bar p=(\bar b,\bar\lambda)\in\operatorname{dom}S_P$
and $\bar x\in S_P(\bar p)$. Set
\[
\bar y:=A\bar x,\qquad
\bar z:=-\bar\lambda^{-1}\nabla_yF(\bar y,\bar b),
\qquad
\bar v:=L^{-T}A^T\bar z,
\]
and let $I(\bar v):=\{i:\langle p_i,\bar v\rangle=\beta_i\}$.
Suppose that:
\begin{enumerate}[label=(\roman*)]
\item $F$ is $\mathcal C^2$ near $(\bar y,\bar b)$, and $F(\cdot,b)$ is convex for all $b$ near $\bar b$; 

\item
\begin{align}
\operatorname{rge}(A^T)\cap L^T\operatorname{ri}P&\neq\emptyset,
\label{e-PLQ-dual-QC}\\
\ker\nabla_{yy}^2F(\bar y,\bar b)\cap AL^{-1}\operatorname{span}\{p_i:i\in I(\bar v)\}&=\{0\}; 
\label{e-PLQ-image-nondeg}
\end{align} 

\item
\begin{equation}\label{e-PLQ-fiber-QC}
AL^{-1}\bigl((P^\infty\cap\ker\Sigma)^\circ\bigr)
=
\operatorname{rge}A;
\end{equation}

\item there exist a neighborhood $U$ of $\bar p$ and a convex set $D\subseteq\operatorname{dom}\Phi_P$ such that
$H_P(p)\subseteq D$ for all $p\in U\cap\operatorname{dom}S_P$, and branch persistence~\eqref{e-branch-persistence} holds on $D$ for \eqref{e-PLQ-CR-decomposition}.
\end{enumerate}
Then $S_P$ is locally Hausdorff Lipschitz continuous
at $\bar p$ relative to $\operatorname{dom}S_P$.
\end{proposition} 

\begin{proof}
We first verify the conditions for $H_P$.
Since $L$ is a linear isomorphism,
\[
R_P^*(w)
=
\frac12\langle
\Sigma L^{-T}w,L^{-T}w
\rangle
+
\delta_P(L^{-T}w).
\]
Thus
$\operatorname{ri}(\operatorname{dom}R_P^*)
=L^T\operatorname{ri}P$, so
\eqref{e-PLQ-dual-QC} is condition~\eqref{QC-H}.

The first-order optimality condition gives
$A^T\bar z\in\partial R_P(\bar x)$,
and hence $\bar v\in P$. Moreover,
\[
(R_P^*\circ A^T)(z)
=
\frac12\langle
\Sigma L^{-T}A^Tz,L^{-T}A^Tz
\rangle
+
\delta_{D_P}(z),
\]
where $D_P:=\{z:L^{-T}A^Tz\in P\}$.
Since $D_P$ is convex polyhedral,
$R_P^*\circ A^T$ is a convex quadratic function
on a polyhedral domain and is therefore
$\mathcal C^2$-cone reducible at $\bar z$. Furthermore,
\[
\partial R_P^*(A^T\bar z)
=
L^{-1}\bigl(\Sigma\bar v+N_P(\bar v)\bigr),
\qquad
N_P(\bar v)
=
\operatorname{cone}\{p_i:i\in I(\bar v)\}.
\]
Consequently,
\[
\operatorname{par}\partial R_P^*(A^T\bar z)
=
L^{-1}\operatorname{span}\{p_i:i\in I(\bar v)\}.
\]
Thus \eqref{e-PLQ-image-nondeg} satisfies the
nondegeneracy condition in Theorem~\ref{L-H-IFT}.
Hence $H_P$ is locally single-valued and Lipschitz continuous. 

Next, by \eqref{e-PLQ-domain},
condition~\eqref{e-PLQ-fiber-QC} is equivalent to
$A(\operatorname{dom}R_P)=\operatorname{rge}A$.
Since $\operatorname{dom}R_P$ is convex and relative
interiors commute with linear mappings,
\[
A\bigl(\operatorname{ri}(\operatorname{dom}R_P)\bigr)
=
\operatorname{ri}\bigl(A(\operatorname{dom}R_P)\bigr)
=
\operatorname{rge}A.
\]
Hence the fiber qualification in
Proposition~\ref{P-Phi-geometry} holds, giving
$\Phi_P(y)=A^{-1}(y)\cap C_{R_P}$.

By \eqref{e-PLQ-CR-decomposition}, assumption~(iv),
and Proposition~\ref{P-union-poly-HLip},
$\Phi_P$ is Hausdorff Lipschitz continuous relative
to $D$. Since $H_P(p)\subseteq D$ for
$p\in U\cap\operatorname{dom}S_P$, the conclusion
follows from Proposition~\ref{P-transfer} applied
to $S_P=\Phi_P\circ H_P$.
\end{proof} 

\begin{remark}\label{rem-PLQ-verification}
The conditions in Proposition~\ref{T-structured-PLQ}
admit concrete sufficient tests. If
$P^\infty\cap\ker\Sigma=\{0\}$, then
$\operatorname{dom}R_P=\mathbb R^n$, so
\eqref{e-PLQ-fiber-QC} holds automatically.
This includes the cases where $P$ is bounded or
$\Sigma$ is positive definite.

The polyhedral pieces $G_\ell$ can be constructed
from the faces of $P$ using
\[
w\in\partial R_P(x)
\Longleftrightarrow
v:=L^{-T}w\in P,\quad
Lx-\Sigma v\in N_P(v).
\]
Since $N_P(v)$ is constant on the relative interior
of each face, this characterization yields a finite
polyhedral decomposition of $\operatorname{gph}\partial R_P$.
The branches $C_\ell$ and their domains $A(C_\ell)$
can then be obtained by polyhedral projection.

Branch persistence on $D$ can be verified by checking
that, for every $\bar y\in D$ and
$\bar x\in\Phi_P(\bar y)$, some branch $C_\ell$
containing $\bar x$ satisfies
$D\cap B(\bar y,\delta)\subseteq A(C_\ell)$
for some $\delta>0$.
In particular, it suffices that $A(C_\ell)$ contain
a relative neighborhood of $\bar y$ in $D$.
\end{remark}

\begin{remark}\label{R-piecewise-l1-relation}
The linear isomorphism assumption on $L$ ensures
the explicit conjugate representation of $R_P$
and the $\mathcal C^2$-cone reducibility of
$R_P^*\circ A^T$. The latter property may fail
without this assumption; see Section~\ref{A-PLQ-nonreducible} of the Appendix.

The structured PLQ class includes $\ell_1$-type
regularizers. For $E\in\mathbb R^{q\times n}$,
let $B_\infty:=\{\xi\in\mathbb R^q:\|\xi\|_\infty\leq1\}$.
Taking $\Sigma=0$, $L=I_n$, and $P=E^TB_\infty$
gives
\[
R_P(x)
=
\sigma_{E^TB_\infty}(x)
=
\sup_{\|\xi\|_\infty\leq1}\langle Ex,\xi\rangle
=
\|Ex\|_1.
\]
For this subclass, Hausdorff Lipschitz continuity
of $\Phi_E$ can be established without branch
persistence~\eqref{e-branch-persistence}, as shown below. 
\end{remark} 
	

\subsubsection{\texorpdfstring{$\ell_1$}{l1}-type regularizers}\label{l1-type regularizers}

Consider the $\ell_1$-type regularizer
\begin{equation}\label{e-l1-type-R}
R_E(x):=\|Ex\|_1,
\qquad E\in\mathbb R^{q\times n}.
\end{equation}
This includes generalized LASSO regularization
\cite{TibshiraniTaylor2011} and reduces to $\|x\|_1$ when $E=I_n$. 

\begin{proposition}\label{P-l1-type}
Let $R_E$ be defined by \eqref{e-l1-type-R}, and let
$S_E$, $H_E$, and $\Phi_E$ denote the corresponding
solution, image, and generalized basis pursuit mappings.
Fix $\bar p=(\bar b,\bar\lambda)\in\operatorname{dom}S_E$
and $\bar x\in S_E(\bar p)$. Set
\[
\bar y:=A\bar x,\qquad
\bar z:=-\bar\lambda^{-1}\nabla_yF(\bar y,\bar b).
\]
Suppose that $F$ is $\mathcal C^2$ near
$(\bar y,\bar b)$ and $F(\cdot,b)$ is convex
for all $b$ near $\bar b$. Define
\[
B_\infty:=\{\xi\in\mathbb R^q:\|\xi\|_\infty\leq1\},
\qquad
\mathcal L_E:=
\operatorname{par}N_{E^TB_\infty}(A^T\bar z).
\]
If
\begin{equation}\label{e-E-image-nondeg}
\ker\nabla_{yy}^2F(\bar y,\bar b)
\cap A\mathcal L_E=\{0\},
\end{equation}
then $S_E$ is locally Hausdorff Lipschitz continuous
at $\bar p$ relative to $\operatorname{dom}S_E$.
\end{proposition}

\begin{proof}
By Proposition~\ref{P-transfer}, it suffices to
establish local single-valued Lipschitz continuity
of $H_E$ and Hausdorff Lipschitz continuity of
$\Phi_E$ relative to its domain.

\medskip
\noindent\textbf{Step 1. Stability of $H_E$.}
Since $R_E=\sigma_{E^TB_\infty}$, we have
$R_E^*=\delta_{E^TB_\infty}$. Moreover,
\[
0\in E^T(\operatorname{ri}B_\infty)
=\operatorname{ri}(E^TB_\infty)
=\operatorname{ri}(\operatorname{dom}R_E^*),
\]
which verifies \eqref{QC-H}.

Furthermore,
\[
R_E^*\circ A^T=\delta_{D_E},
\qquad
D_E:=\{z:A^Tz\in E^TB_\infty\}.
\]
Since $D_E$ is convex polyhedral,
$R_E^*\circ A^T$ is $\mathcal C^2$-cone reducible
at $\bar z$; see
\cite[Example~3.139]{BonnansShapiro2000}.

The first-order optimality condition gives
$A^T\bar z\in\partial R_E(\bar x)$, and hence
\[
\operatorname{par}\partial R_E^*(A^T\bar z)
=
\operatorname{par}N_{E^TB_\infty}(A^T\bar z)
=
\mathcal L_E.
\]
Thus \eqref{e-E-image-nondeg} is the nondegeneracy
condition in Theorem~\ref{L-H-IFT}.
Consequently, $H_E$ is locally single-valued
and Lipschitz continuous.

\medskip
\noindent\textbf{Step 2. Stability of $\Phi_E$.}
Define $Gx:=(Ax,Ex)$ and
$\mathcal V:=\operatorname{rge}G$. For
$y\in\operatorname{dom}\Phi_E$, let
\[
\Gamma_E(y)
:=
\arg\min_u\{\|u\|_1:(y,u)\in\mathcal V\},
\]
and define
$\Psi(y,u):=\{x:Ax=y,\ Ex=u\}$
for $(y,u)\in\mathcal V$. Then
\begin{equation}\label{e-PhiE-decomp}
\Phi_E(y)
=
\bigcup_{u\in\Gamma_E(y)}\Psi(y,u).
\end{equation}

To characterize $\Gamma_E$, choose a matrix
$C=(C_y,C_u)$ whose rows form a basis of
$\mathcal V^\perp$, with
$C_y\in\mathbb R^{r\times m}$,
$C_u\in\mathbb R^{r\times q}$, and
$r=\dim\mathcal V^\perp$.
Since $\mathcal V=\ker C$,
\[
\Gamma_E(y)
=
\arg\min_u\{\|u\|_1:C_uu=-C_yy\}.
\]
For $y\in\operatorname{dom}\Phi_E$,
$-C_yy\in\operatorname{rge}C_u$.
By \cite[Theorem~3.1]{meng2024lip},
$\Gamma_E$ is Hausdorff Lipschitz continuous
relative to $\operatorname{dom}\Phi_E$,
with some modulus $L_\Gamma>0$.

Next, let $K_G:=\ker G$ and
$B_G:=G|_{K_G^\perp}:K_G^\perp\to\mathcal V$.
Since $B_G$ is a linear isomorphism,
\[
\Psi(y,u)=B_G^{-1}(y,u)+K_G.
\]
For $(y_i,u_i)\in\mathcal V$, $i=1,2$, we have
$d_H(\Psi(y_1,u_1),\Psi(y_2,u_2))\leq L_\Psi\|(y_1-y_2,u_1-u_2)\|$, where $L_\Psi:=\|B_G^{-1}\|$. 

Define
$\widehat\Gamma_E(y)
:=\{(y,u):u\in\Gamma_E(y)\}$.
The Hausdorff Lipschitz estimate for $\Gamma_E$
implies
\[
d_H\bigl(
\widehat\Gamma_E(y_1),
\widehat\Gamma_E(y_2)
\bigr)
\leq
\sqrt{1+L_\Gamma^2}\,\|y_1-y_2\|.
\]
Since
$\Phi_E=\Psi\circ\widehat\Gamma_E$,
Proposition~\ref{P-transfer} yields
\[
d_H\bigl(\Phi_E(y_1),\Phi_E(y_2)\bigr)
\leq
L_\Psi\sqrt{1+L_\Gamma^2}\,
\|y_1-y_2\|,
\]
for all $y_1,y_2\in\operatorname{dom}\Phi_E$.
Hence $\Phi_E$ is Hausdorff Lipschitz continuous
relative to its domain.

Finally, applying Proposition~\ref{P-transfer}
to $S_E=\Phi_E\circ H_E$ completes the proof.
\end{proof} 

\begin{remark}\label{R-l1-no-persistence}
Proposition~\ref{P-l1-type} establishes Hausdorff Lipschitz continuity of $\Phi_E$ without branch persistence. Introducing $u=Ex$ decomposes $\Phi_E$ into an $\ell_1$-minimization mapping $\Gamma_E$ and an affine fiber mapping $\Psi$. Their Hausdorff Lipschitz estimates yield stability of $\Phi_E$ via Proposition~\ref{P-transfer}, without a polyhedral decomposition of $C_{R_E}$. 
\end{remark}

\begin{corollary}\label{cor-l1}
Under the assumptions of Proposition~\ref{P-l1-type} with $E=I_n$,
the solution mapping $S$ for $R(x)=\|x\|_1$ is locally
Hausdorff Lipschitz continuous at $\bar p$ relative to $\operatorname{dom}S$.
\end{corollary} 

\begin{proof}
Apply Proposition~\ref{P-l1-type} with $E=I_n$.
\end{proof}

\begin{remark}\label{R-l1-weak-strong}
In the setting of Corollary~\ref{cor-l1}, let
$J:=\{i\in[n]:|(A^T\bar z)_i|=1\}$.
Since $R^*=\delta_{B_\infty}$, we have
\[
\mathcal L
=
\operatorname{par}\partial R^*(A^T\bar z)
=
\{d\in\mathbb R^n:d_{J^c}=0\},
\qquad
A\mathcal L=\operatorname{rge}(A_J).
\]
Hence condition~\eqref{e-E-image-nondeg} becomes
\[
\ker\nabla_{yy}^2F(\bar y,\bar b)
\cap\operatorname{rge}(A_J)=\{0\},
\]
while \eqref{e-A-injectivity} reduces to $\ker A_J=\{0\}$.
By Theorem~\ref{T-primal-image-ND}, primal-space nondegeneracy is equivalent to image-space nondegeneracy together with this full-column-rank condition. 

When $F(y,b)=\varphi(y-b)$, these conditions correspond
to the weak and strong conditions in
\cite{HuYaoZhang2025}, respectively.
The former allows Hausdorff Lipschitz stability of
set-valued solution mappings, while the latter also
imposes the rank condition associated with local
single-valued Lipschitz stability.
Thus the composite representation permits Hausdorff
Lipschitz stability without requiring uniqueness
within the fibers. Corollary~\ref{cor-l1} recovers
the corresponding result in
\cite[Theorem~2(ii)]{HuYaoZhang2025}
without the finite sign-pattern decomposition
used there.
\end{remark} 

	
\section{Numerical experiments}
\label{sec4}

We illustrate how the composite representation can be used with
existing optimization methods to recover complete solution sets.
When the reduced image problem has a unique minimizer
$H(b,\lambda)=\{y^*\}$, the representation
$S(b,\lambda)=\Phi(y^*)$ allows recovery of the complete
solution set from the optimal image and its associated fiber. 

We consider two-dimensional examples with $\ell_1$ regularization and a nonpolyhedral graph of the generalized basis pursuit mapping. A three-dimensional extension and a sparse-observation total variation (TV) reconstruction experiment are provided in Sections~\ref{app:three_dimensional} and~\ref{app:tv_numerical} of the Appendix, respectively. The latter illustrates dimensional reduction from $257$ primal variables to $17$ image variables, complete solution-set recovery, and the effect of regularization on nonuniqueness. 

Throughout, we set $b=1$ and use three
data-fidelity functions: quartic, quadratic, and pseudo-Huber,
\[
F_1(y,b)=\frac14(y-b)^4,
\qquad
F_2(y,b)=\frac12(y-b)^2,
\]
and
\[
F_3(y,b)
=
\delta\sqrt{1+\delta^2}
\left(
\sqrt{1+\left(\frac{y-b}{\delta}\right)^2}-1
\right),
\qquad \delta=0.2.
\]
All three satisfy $|\nabla_yF_j(0,1)|=1$, so the reduced
problems with regularizer $|y|$ share the threshold
$H_j(1,\lambda)=\{0\}$ for $\lambda\geq1$. 


\subsection{Example with the \texorpdfstring{$\ell_1$}{l1} regularizer}

Consider $R(x)=\|x\|_1$ and
$A=(1,1)\in\mathbb R^{1\times2}$. Since
$V(y)=\min_{x_1+x_2=y}\|x\|_1=|y|$, the reduced
image problems are
\[
H_j(1,\lambda)
=
\arg\min_{y\in\mathbb R}
\{F_j(y,1)+\lambda|y|\},
\qquad j=1,2,3.
\]
Each problem has a unique minimizer $y_j^*(\lambda)$
by strict convexity. The first-order optimality
condition gives $H_j(1,\lambda)=\{y_j^*(\lambda)\}$,
where
\[
y_1^*(\lambda)
=
\begin{cases}
1-\lambda^{1/3}, & 0<\lambda<1,\\
0, & \lambda\geq1,
\end{cases}
\qquad
y_2^*(\lambda)
=
\begin{cases}
1-\lambda, & 0<\lambda<1,\\
0, & \lambda\geq1,
\end{cases}
\]
and
\[
y_3^*(\lambda)
=
\begin{cases}
\displaystyle
1-\frac{\delta\lambda}
{\sqrt{1+\delta^2-\lambda^2}},
& 0<\lambda<1,\\[2mm]
0, & \lambda\geq1,
\end{cases}
\qquad \delta=0.2.
\]
Since $\nabla_{yy}^2F_j(y_j^*(\lambda),1)>0$ for all $\lambda>0$,
Proposition~\ref{P-l1-type} ensures that $H_j$ is locally
single-valued and Lipschitz continuous near $(1,\lambda)$.
Figure~\ref{fig:num_H_Phi_2d}(a) compares the explicit paths
with interior-point (IP) solutions; see
Section~\ref{app:image space} of the Appendix. 

Although $H_j$ is single-valued, the primal solution
mapping is set-valued for $0<\lambda<1$ because
$\ker A=\operatorname{span}\{(1,-1)\}$.
For $y\geq0$, minimizing $\|x\|_1$ subject to
$x_1+x_2=y$ gives
\[
\Phi(y)
=
\{x\in\mathbb R_+^2:x_1+x_2=y\}.
\]
Equivalently, with $d=x_1-x_2$, the fiber satisfies
$-y\leq d\leq y$. Thus $\Phi(y)$ is a line segment
for $y>0$ and reduces to $\{0\}$ at $y=0$;
see Figure~\ref{fig:num_H_Phi_2d}(b).

\begin{figure}[ht]
    \centering
    \begin{subfigure}[t]{0.60\textwidth}
        \centering
        \includegraphics[width=\textwidth]{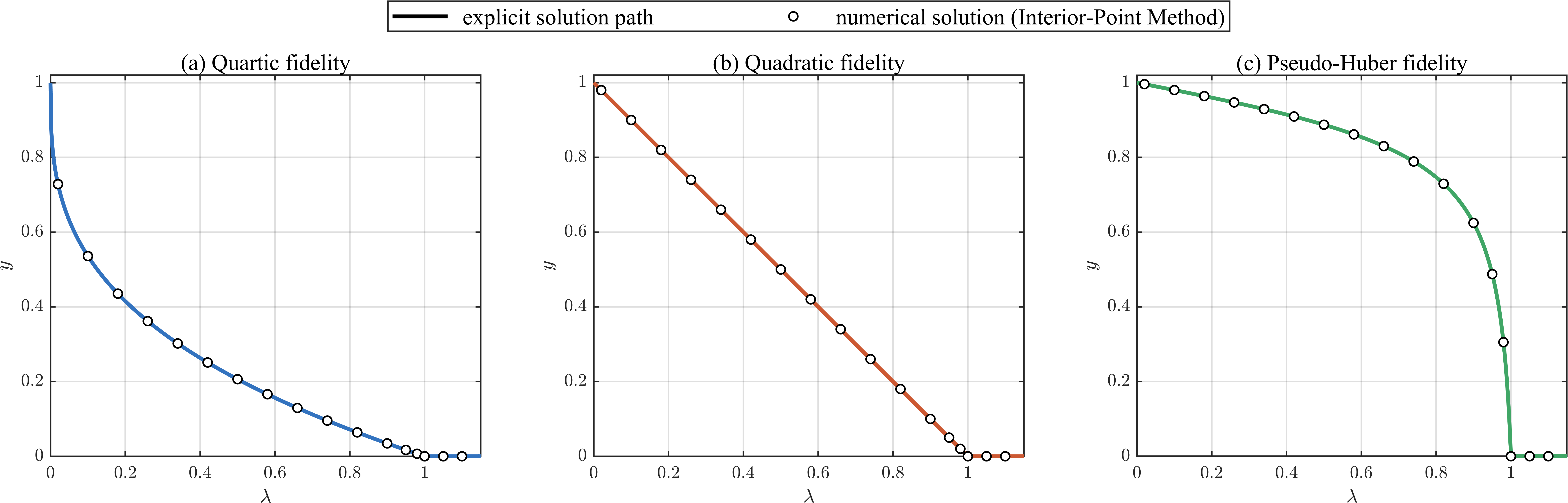}
        \caption{}
        \label{fig:num_H}
    \end{subfigure}
    \hspace{0.05\textwidth}
    \begin{subfigure}[t]{0.105\textwidth}
        \centering
        \includegraphics[width=\textwidth]{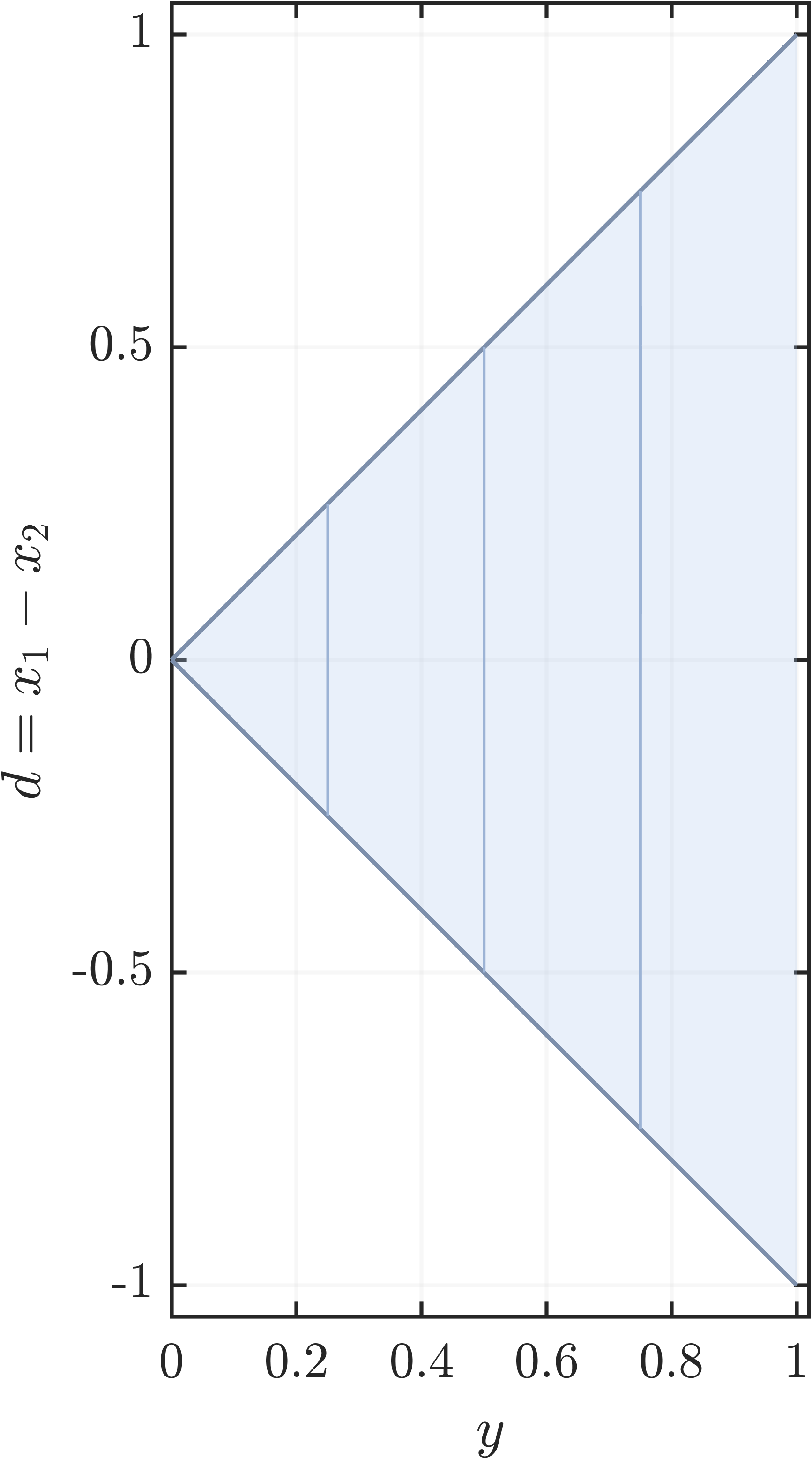}
        \caption{}
        \label{fig:num_Phi_2d}
    \end{subfigure}
    \hspace{0.05\textwidth}
    \begin{subfigure}[t]{0.144\textwidth}
        \centering
        \includegraphics[width=\textwidth]{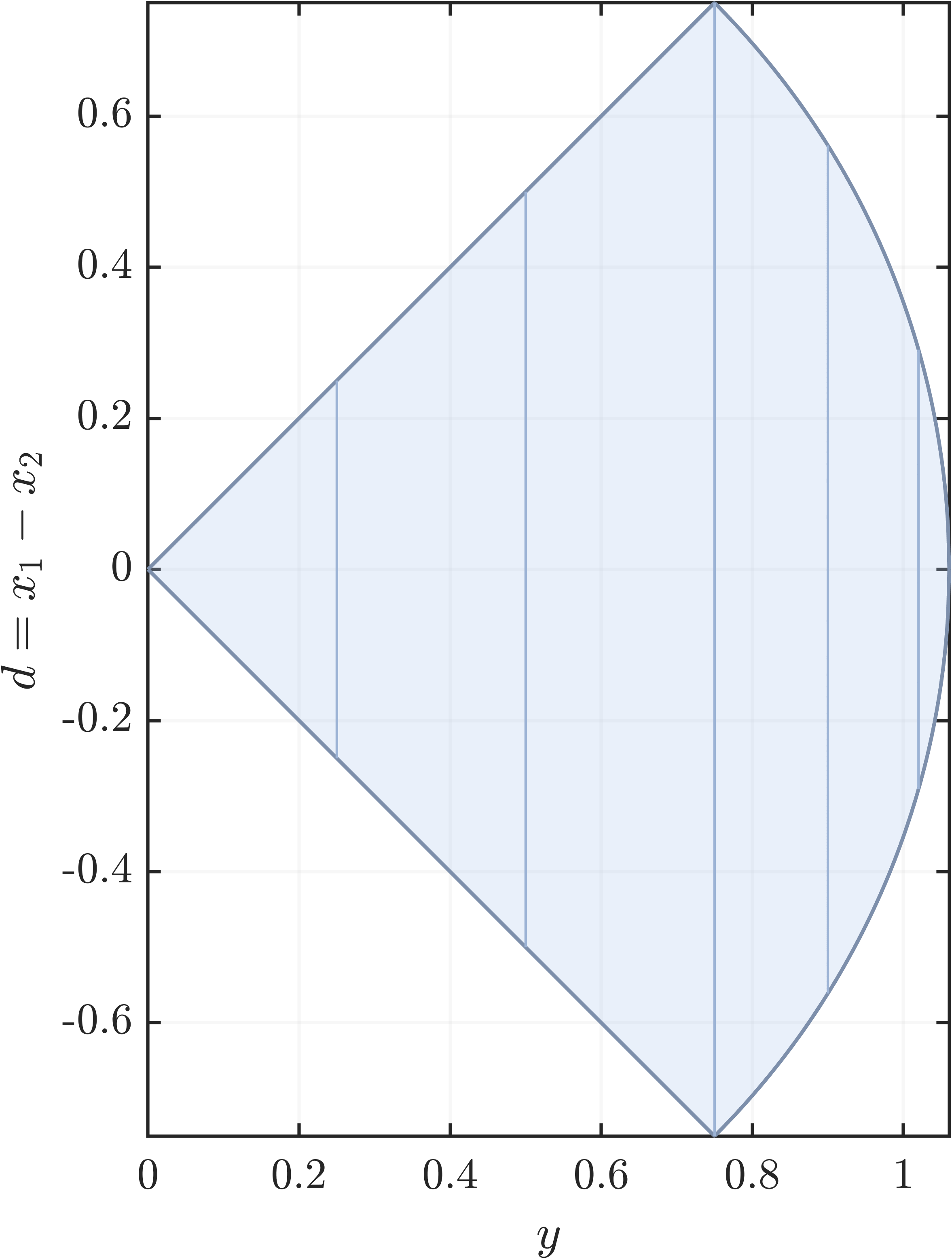}
        \caption{}
        \label{fig:num_Phi_2d_nonpoly}
    \end{subfigure}
    \caption{Components of the composite representation.
(a) Optimal image paths $H$ for the three data-fidelity functions and corresponding numerical solutions.
(b) Polyhedral graph of $\Phi$ for the $\ell_1$ regularizer.
(c) Nonpolyhedral graph of $\Phi$ for the regularizer in Section~\ref{example_nonpolyhedral}.} 
    \label{fig:num_H_Phi_2d}
\end{figure}

The composite representation therefore yields
\[
S_j(1,\lambda)
=
\Phi\bigl(y_j^*(\lambda)\bigr)
=
\begin{cases}
\{x\in\mathbb R_+^2:x_1+x_2=y_j^*(\lambda)\},
& 0<\lambda<1,\\
\{0\}, & \lambda\geq1,
\end{cases}
\qquad j=1,2,3.
\]
By Corollary~\ref{cor-l1}, these set-valued mappings are locally Hausdorff Lipschitz continuous. 

We also solve the original problem using FISTA, IP, and ADMM from ten diverse initializations, independently for each $\lambda$. Figure~\ref{fig:num_S_direct_2d} compares the complete solution sets with direct numerical selections. Solver settings and additional initialization results appear in Sections~\ref{app:solver} and~\ref{app:second initialization} of the Appendix. 

\begin{figure}[h]
    \centering
    \includegraphics[width=0.68\textwidth]{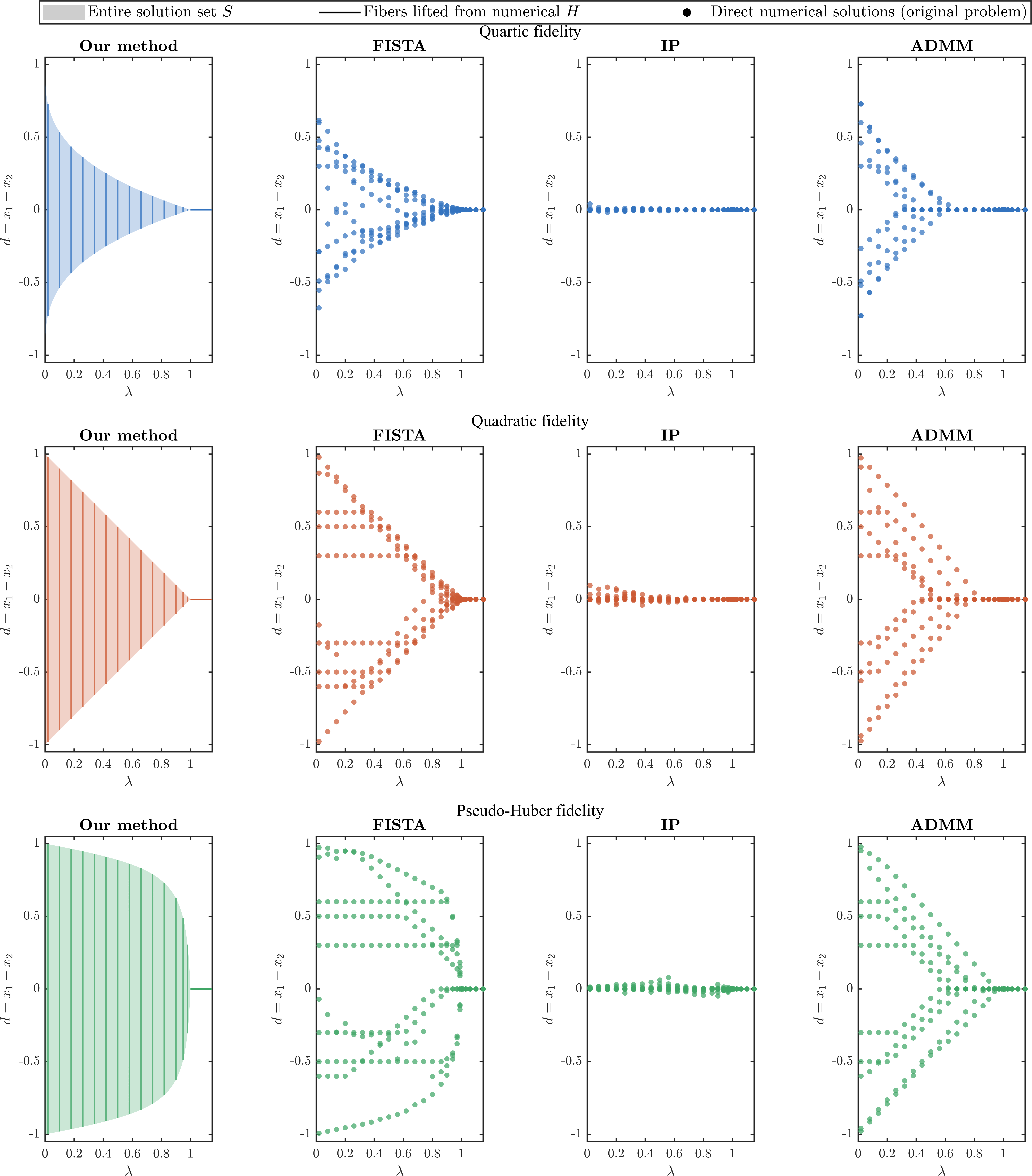}
\caption{Composite representation and direct numerical solutions
in the $\ell_1$-regularized example. Shaded regions show the
complete solution sets, colored fibers are recovered from numerical
optimal images, and points denote solutions returned by FISTA,
IP, and ADMM applied directly to the original problem.} 
    \label{fig:num_S_direct_2d}
\end{figure} 


\subsection{Example with nonpolyhedral \texorpdfstring{$\operatorname{gph}\Phi$}{gph Phi}}
\label{example_nonpolyhedral}

We next consider a two-dimensional example with nonpolyhedral
$\operatorname{gph}\Phi$:
\[
\min_{\|x\|_2\leq r}
\{F_j(Ax,1)+\lambda\|x\|_1\},
\qquad j=1,2,3.
\]
Related models combining sparsity penalties with Euclidean-norm
constraints arise in sparse principal component analysis
\cite{journee2010generalized}. This problem has the
form~\eqref{P1} with
$A=(1,1)$, $R(x)=\|x\|_1+\delta_{rB_2}(x)$, and $r=0.75$,
where $B_2:=\{x\in\mathbb R^2:\|x\|_2\leq1\}$.

The generalized basis pursuit value function is
\[
V(y)
=
\begin{cases}
|y|, & |y|\leq\sqrt{2}\,r,\\
+\infty, & |y|>\sqrt{2}\,r.
\end{cases}
\]
Indeed, $\|x\|_1\geq|x_1+x_2|=|y|$, with equality
attained at $x=(y/2,y/2)$ whenever $|y|\leq\sqrt{2}\,r$.
Since $\sqrt{2}\,r>1$ and $0\leq y_j^*(\lambda)\leq1$,
the optimal image paths coincide with those of the preceding
$\ell_1$ example; see Figure~\ref{fig:num_H_Phi_2d}(a).

\begin{figure}[h]
    \centering
    \includegraphics[width=0.70\textwidth]{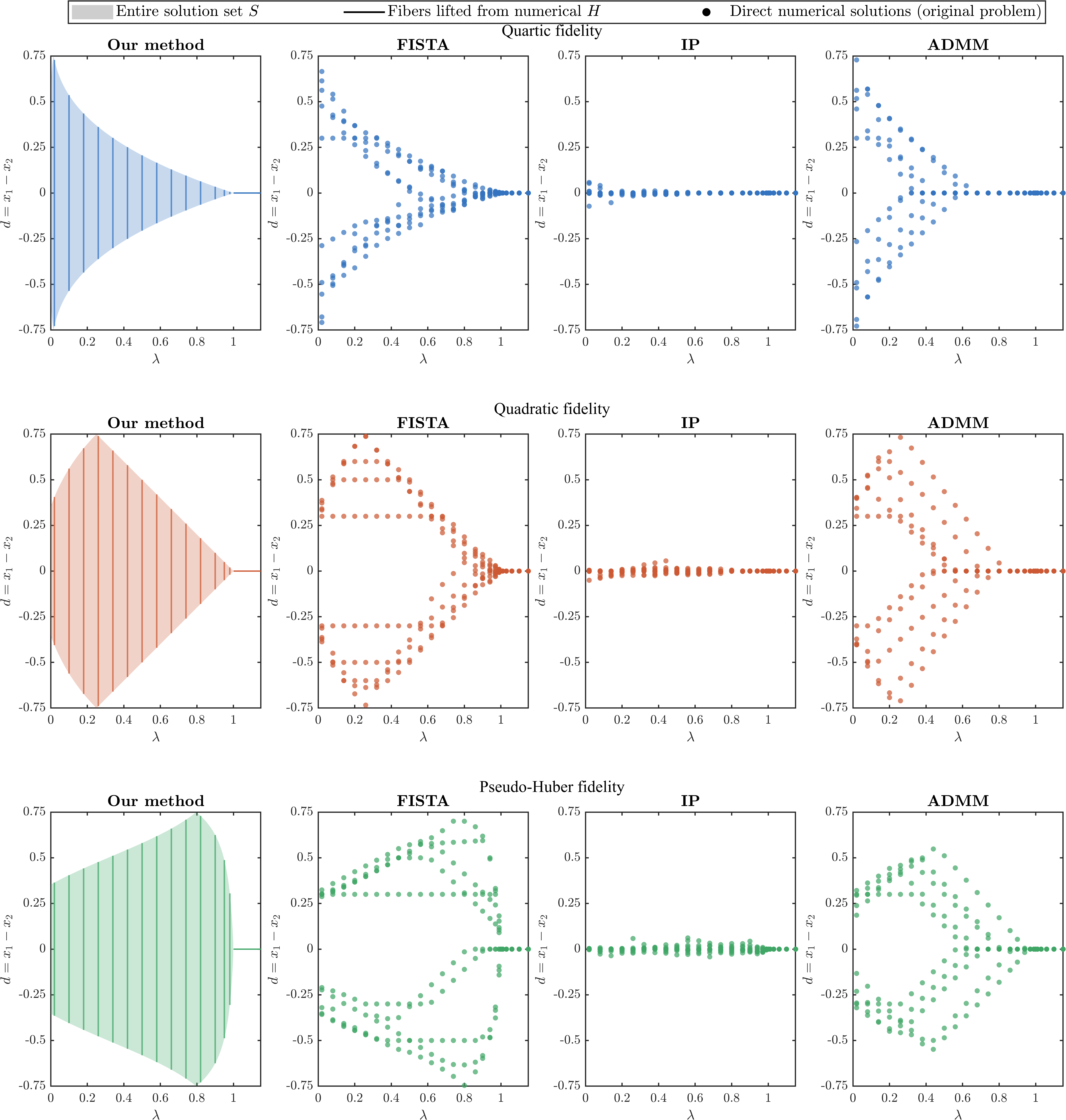}
    \caption{Composite representation and direct numerical solutions
in the nonpolyhedral example. Shaded regions represent the
complete solution sets, colored fibers are recovered from
numerically computed optimal images via the composite
representation, and points denote solutions obtained by
applying FISTA, IP, and ADMM directly to the original problem.} 
    \label{fig:num_S_direct_2d_nonpoly}
\end{figure}

The difference lies in the geometry of $\Phi$.
For fixed $y=x_1+x_2$, parameterize the affine fiber by $d=x_1-x_2$, giving
$x_1=(y+d)/2$ and $x_2=(y-d)/2$. 
For $y\geq0$, the minimum-$\ell_1$ condition gives
$|d|\leq y$, while the ball constraint gives
$|d|\leq\sqrt{2r^2-y^2}$. Thus the fibers satisfy
\[
|d|
\leq
\begin{cases}
y, & 0\leq y\leq r,\\
\sqrt{2r^2-y^2}, & r<y\leq\sqrt{2}\,r.
\end{cases}
\]
The boundary of $\operatorname{gph}\Phi$ therefore contains
linear segments and circular arcs, making it nonpolyhedral;
see Figure~\ref{fig:num_H_Phi_2d}(c).

The complete solution mappings are given by
$S_j(1,\lambda)=\Phi(y_j^*(\lambda))$, $j=1,2,3$.
Using the same numerical settings as in the preceding example,
we compare these mappings with direct FISTA, IP, and ADMM
solutions in Figure~\ref{fig:num_S_direct_2d_nonpoly}.
Additional results appear in Section~\ref{app:second initialization}
of the Appendix. 

	
\section{Conclusion}\label{sec5}

We studied the geometry and stability of solution sets
in regularized composite optimization through an
image--fiber decomposition. By analyzing the image
mapping and generalized basis pursuit mapping separately,
we obtained sufficient conditions for local Hausdorff
Lipschitz stability without requiring uniqueness of
the original minimizers. Our comparison with direct
primal-space analysis identified an additional
injectivity requirement in the primal-space conditions.
For classical $\ell_1$ regularization, the framework
also recovers an existing stability result without
a finite sign-pattern decomposition. Numerical examples
illustrated the recovery of complete solution sets,
including a case with nonpolyhedral fiber geometry.

Although the composite representation does not require
convexity, our geometric and stability results rely
on convexity and additional regularity assumptions.
Further directions include extending the analysis
to nonconvex settings and establishing stability
under more general nonpolyhedral structures.
Developing numerical methods that exploit the
image--fiber decomposition without explicitly
characterizing the generalized basis pursuit mapping
is another direction. Extensions to
infinite-dimensional spaces, where attainment and compactness require additional care, also merit investigation.

\section*{Acknowledgments}
The authors are grateful to Xiyin Zheng and Chunhai Hu for helpful
discussions on this research topic, particularly regarding the
geometric characterization of the generalized basis pursuit mapping. 
	
	\bibliographystyle{plainnat}
	\bibliography{references}

\clearpage
\appendix
\numberwithin{figure}{section}
\numberwithin{table}{section}

\section{Examples illustrating branch persistence}
\label{sec:branch-persistence-examples}

The following examples illustrate the role of branch persistence.
A disappearing branch may destroy Hausdorff continuity,
but branch persistence is not necessary for Hausdorff
Lipschitz continuity. 

\begin{example}\label{E-union-not-HLip}
Let $A:\mathbb R^2\to\mathbb R$ be defined by
$A(x_1,x_2)=x_1$, and set
\[
C_1:=\mathbb R\times\{0\},
\qquad
C_2:=\{0\}\times[0,1],
\qquad
C:=C_1\cup C_2.
\]
Define $\Psi(y):=A^{-1}(y)\cap C$.
Then $\Psi(y)=\{(y,0)\}$ for $y\neq0$, whereas
$\Psi(0)=\{0\}\times[0,1]$. Hence, for $y\neq0$,
\[
d_H\bigl(\Psi(y),\Psi(0)\bigr)
\geq
\operatorname{dist}\bigl((0,1),\Psi(y)\bigr)
=
\sqrt{1+y^2},
\]
so $\Psi$ is not Hausdorff continuous at $0$.

The discontinuity arises from $C_2$, whose domain
$A(C_2)=\{0\}$ excludes every nonzero perturbation.
At $(\bar y,\bar x)=(0,(0,1))$, only $C_2$
contains $\bar x$, so the analogous branch-persistence
condition fails on $D=(-1,1)$.
In contrast, on $D=(0,1)$, only $C_1$ is relevant,
and $\Psi$ is Hausdorff Lipschitz continuous.
\end{example}

This example concerns a general affine-slice mapping,
rather than a lifting set $C_R$ generated by a convex
regularizer. It shows that Hausdorff Lipschitz continuity
of individual polyhedral branches need not extend
to their union. The next example shows that branch
persistence is not necessary, even for $\ell_1$ regularization.

\begin{example}\label{E-l1-no-persistence}
Let $A=(1,1)\in\mathbb R^{1\times2}$ and
$R(x)=\|x\|_1$. Since
$\operatorname{rge}(A^T)=\operatorname{span}\{(1,1)\}$,
\[
C_R
=
(\partial R)^{-1}\bigl(\operatorname{rge}(A^T)\bigr)
=
\mathbb R_+^2\cup\mathbb R_-^2
=:C_+\cup C_-.
\]
Indeed, the components of $x\in C_R$ must have
the same weak sign. The branch domains are
$A(C_+)=[0,+\infty)$ and $A(C_-)=(-\infty,0]$.

At $(\bar y,\bar x)=(0,0)$, both branches contain
$\bar x$, but neither remains feasible for all
nearby $y\in D=\mathbb R$. Thus branch
persistence~\eqref{e-branch-persistence} fails on $D$.

Nevertheless, the generalized basis pursuit mapping is
\[
\Phi(y)
=
\{y(t,1-t):t\in[0,1]\},
\qquad y\in\mathbb R.
\]
Matching points with the same $t$ yields
\[
d_H\bigl(\Phi(y_1),\Phi(y_2)\bigr)
\leq |y_1-y_2|,
\qquad y_1,y_2\in\mathbb R.
\]
Hence $\Phi$ is Hausdorff Lipschitz continuous
despite the failure of branch persistence.
\end{example}

These examples show that branch persistence is sufficient
but not necessary for Hausdorff Lipschitz continuity.
In particular, polyhedral branches may switch without
destroying stability. For $R_E(x)=\|Ex\|_1$,
Section~\ref{l1-type regularizers} of the main paper
establishes Hausdorff Lipschitz continuity without
requiring branch persistence.


\section{Failure of \texorpdfstring{$\mathcal C^2$}{C2}-cone reducibility without an isomorphism}
\label{A-PLQ-nonreducible}

We provide a counterexample showing that the
$\mathcal C^2$-cone reducibility of $R_P^*\circ A^T$
required in Theorem~\ref{L-H-IFT} may fail
without the linear isomorphism assumption on $L$.

Consider the linearly composite PLQ regularizer
\[
R_{\mathcal P}(x)
=
\theta_{\mathcal P,\Sigma}(Lx)
=
\sup_{v\in\mathcal P}
\left\{
\langle Lx,v\rangle
-\frac12\langle\Sigma v,v\rangle
\right\},
\]
where
\[
\mathcal P
:=
\{(u_1,u_2,s)\in\mathbb R^3:
s\geq|u_1|+|u_2|\},
\]
$\Sigma=I_3$, and
$L:\mathbb R^2\to\mathbb R^3$ is defined by
$Lx=(x_1,x_2,0)^T$.
Thus $L$ is injective but not surjective.

For fixed $u=(u_1,u_2)$, the supremum over $s$
is attained at $s=|u_1|+|u_2|$. Hence
\[
R_{\mathcal P}(x)
=
\sup_{u\in\mathbb R^2}
\left\{
\langle x,u\rangle
-u_1^2-u_2^2-|u_1u_2|
\right\}
=
f^*(x),
\]
where
\[
f(u)
:=
u_1^2+u_2^2+|u_1u_2|
=
\frac12\bigl(\|u\|_2^2+\|u\|_1^2\bigr).
\]
Since $f$ is proper, closed, and convex,
$R_{\mathcal P}^*=f^{**}=f$.
Taking $A=I_2$ gives
$R_{\mathcal P}^*\circ A^T=f$.

To identify a reference point for Theorem~\ref{L-H-IFT},
let $F(y,b)=\frac12\|y-b\|^2$, fix $\bar b=0$
and $\bar\lambda>0$, and set $\bar x=\bar y=0$.
Since $R_{\mathcal P}\geq0$ and
$R_{\mathcal P}(0)=0$, we have
$\bar x\in S_{\mathcal P}(\bar b,\bar\lambda)$.
Moreover,
$\bar z=-\bar\lambda^{-1}\nabla_yF(\bar y,\bar b)=0$.
We show that $f$ is not $\mathcal C^2$-cone
reducible at the origin.

Let $\mathcal C:=\operatorname{epi}f$.
Since $f(\tau d)=\tau^2f(d)$ for $\tau\geq0$,
we have $f'(0;d)=0$ for every $d\in\mathbb R^2$.
Consequently,
\[
T_{\mathcal C}(0,0)
=
\mathbb R^2\times\mathbb R_+.
\]

Suppose, to the contrary, that $\mathcal C$ is
$\mathcal C^2$-cone reducible at the origin.
Then there exist a Euclidean space $W$,
a pointed closed convex cone $\mathcal Q\subseteq W$,
and a $\mathcal C^2$ mapping $h$ defined near
the origin such that $h(0,0)=0$,
$Dh(0,0)$ is surjective, and
\[
\mathcal C
=
\{(u,t):h(u,t)\in\mathcal Q\}
\]
locally around the origin.

The tangent-cone formula gives
$T_{\mathcal C}(0,0)=Dh(0,0)^{-1}(\mathcal Q)$.
Since $\mathcal Q$ is pointed,
\[
\operatorname{lin}T_{\mathcal C}(0,0)
=
\ker Dh(0,0)
=
\mathbb R^2\times\{0\}.
\]
Thus $Dh(0,0)$ has rank one and, by surjectivity,
$\dim W=1$.

Since $T_{\mathcal C}(0,0)$ is not a subspace,
$\mathcal Q\neq\{0\}$.
After identifying $W$ with $\mathbb R$ and
changing the sign of $h$ if necessary,
we may assume $\mathcal Q=\mathbb R_+$.
Hence $\mathcal C=\{(u,t):h(u,t)\geq0\}$ locally.

Because $(0,0,1)\notin\ker Dh(0,0)$,
the partial derivative of $h$ with respect to $t$
is nonzero at the origin.
The implicit function theorem yields a
$\mathcal C^2$ function $g$ near $0$ such that
$h(u,t)=0$ if and only if $t=g(u)$.
Thus the local boundary of $\mathcal C$ is
the graph of $g$. Since
$\mathcal C=\operatorname{epi}f$ and $f$ is continuous,
we obtain $f=g$ near the origin.
In particular, $f$ would be $\mathcal C^2$ at $0$.

However, $\nabla f(0)=0$, and the homogeneity
$f(\tau d)=\tau^2f(d)$ would then imply,
by a second-order Taylor expansion,
\[
f(d)
=
\frac12\langle\nabla^2f(0)d,d\rangle,
\qquad d\in\mathbb R^2.
\]
Thus $f$ would be a quadratic form and satisfy
the parallelogram identity. This contradicts
\[
f((1,1)^T)+f((1,-1)^T)
=
6\neq4
=
2f(e_1)+2f(e_2).
\]
Therefore $f$ is not $\mathcal C^2$ at the origin,
contradicting the assumed cone reducibility.

Consequently, $R_{\mathcal P}^*\circ A^T=f$
is not $\mathcal C^2$-cone reducible at the origin,
even with $A=I_2$. Hence the linear isomorphism
assumption on $L$ cannot be omitted in general
without further conditions. 


\section{Additional numerical details}
\label{app:numerical_details}

This section reports numerical settings and additional results
for the experiments in the main paper. The two- and
three-dimensional examples use $m=1$ and $n\in\{2,3\}$;
the sparse-observation TV experiment uses $m=17$ and $n=257$.
Each problem is solved independently
for every parameter value and initialization, without warm starts.

\subsection{Image-space computations}
\label{app:image space}

We numerically verify the explicit optimal image paths
$H_j(1,\lambda)=\{y_j^*(\lambda)\}$ derived in the main paper.
For Figure~\ref{fig:num_H_Phi_2d}(a), the interior-point
method (IP) is applied at
\[
\lambda\in
\{
0.02,0.10,0.18,0.26,0.34,0.42,0.50,0.58,0.66,
0.74,0.82,0.90,0.95,0.98,1.00,1.05,1.10
\}.
\]
The image-space initializations are
\[
y^0\in
\{-2,-1,-0.6,-0.5,-0.3,0,0.3,0.5,0.6,1,1.2,2\}.
\]
These combine generic starting points with image values
induced by the primal initializations described below.
Each $(\lambda,y^0)$ pair is solved independently.
Figure~\ref{fig:num_H_Phi_2d}(a) shows the solutions
from $y^0=1$; the remaining initializations provide
a multistart consistency check. All computed solutions
agree with the explicit paths within the prescribed tolerance.

For the nonpolyhedral example, the image problem is
restricted to $[-\sqrt{2}\,r,\sqrt{2}\,r]$, where $r=0.75$.
Initializations outside this interval are projected onto it.
Since $\sqrt{2}\,r>1$ and $0\leq y_j^*(\lambda)\leq1$,
the restriction does not affect the optimal image paths.


\subsection{Direct-solver settings and initializations}
\label{app:solver}

All direct-solver experiments use the same parameter grid:
\[
\begin{aligned}
\lambda\in\{&
0.02,0.08,0.14,0.20,0.26,0.32,0.38,0.44,0.50,0.56,0.62,0.68,\\
&0.74,0.80,0.86,0.90,0.94,0.97,0.99,1.00,1.01,1.03,1.06,1.10,1.15
\}.
\end{aligned}
\]
The finer grid near $\lambda=1$ captures the transition
from nontrivial solution fibers to a singleton.

\begin{table}[ht]
\centering
\caption{Main settings of the direct numerical methods.}
\label{tab:solver_parameters}
\renewcommand{\arraystretch}{1.15}
\begin{tabular}{@{}lp{0.78\linewidth}@{}}
\hline
Method & Main settings \\
\hline
FISTA
& Initial Lipschitz estimate $L_0=1$,
backtracking factor $2$, tolerance $10^{-10}$ \\
IP
& Optimality tolerance $10^{-11}$,
constraint tolerance $10^{-10}$ \\
ADMM
& Penalty parameter $\rho=1$,
absolute/relative tolerances $10^{-10}/10^{-9}$ \\
\hline
\end{tabular}
\end{table}

Table~\ref{tab:solver_parameters} summarizes the solver settings. FISTA uses backtracking without restart or a monotonicity
safeguard, with soft-thresholding for the $\ell_1$ examples.
For $R(x)=\|x\|_1+\delta_{rB_2}(x)$, FISTA and the ADMM
$z$-update use the corresponding constrained proximal mapping,
while IP explicitly incorporates the Euclidean-ball constraint.
Other solver settings remain unchanged across the experiments.

\medskip
\noindent\textbf{Two-dimensional initializations.}

Both two-dimensional examples use $A=(1,1)$ and the same
initialization sets. The first contains cross-quadrant
and multi-scale starting points:
\[
X_A^0
=
\begin{pmatrix}
0.40&-0.10&-0.40&0.10&0.80&-0.20&-0.80&0.20&1.60&-0.40\\
0.10&0.40&-0.10&-0.40&0.20&0.80&-0.20&-0.80&0.40&1.60
\end{pmatrix}.
\]
For the nonpolyhedral example, initial feasibility is not
required; the Euclidean-ball constraint is handled by the solvers.

The second initialization set is
\[
X_B^0
=
\begin{pmatrix}
0.33&0.36&0.39&0.42&0.46&0.54&0.58&0.61&0.64&0.67\\
0.67&0.64&0.61&0.58&0.54&0.46&0.42&0.39&0.36&0.33
\end{pmatrix}.
\]
Each column satisfies $Ax^0=\|x^0\|_1=1$.
Thus the initial points differ only along
$\ker A=\operatorname{span}\{(1,-1)\}$, preserving
the image and regularizer value. Moreover,
$\|x^0\|_2\leq0.75$ for every column, so this
set is feasible for the nonpolyhedral example. 


\subsection{Additional two-dimensional direct-solver results}
\label{app:second initialization}

Figures~\ref{fig:app_direct_2d_B} and
\ref{fig:app_nonpoly_direct_B} report numerical selections
from the second initialization set for the $\ell_1$
and nonpolyhedral examples, respectively.
Together with the main-paper results, they illustrate
how direct numerical selections may depend on initialization.

\begin{figure}[ht]
    \centering
    \includegraphics[width=0.78\textwidth]{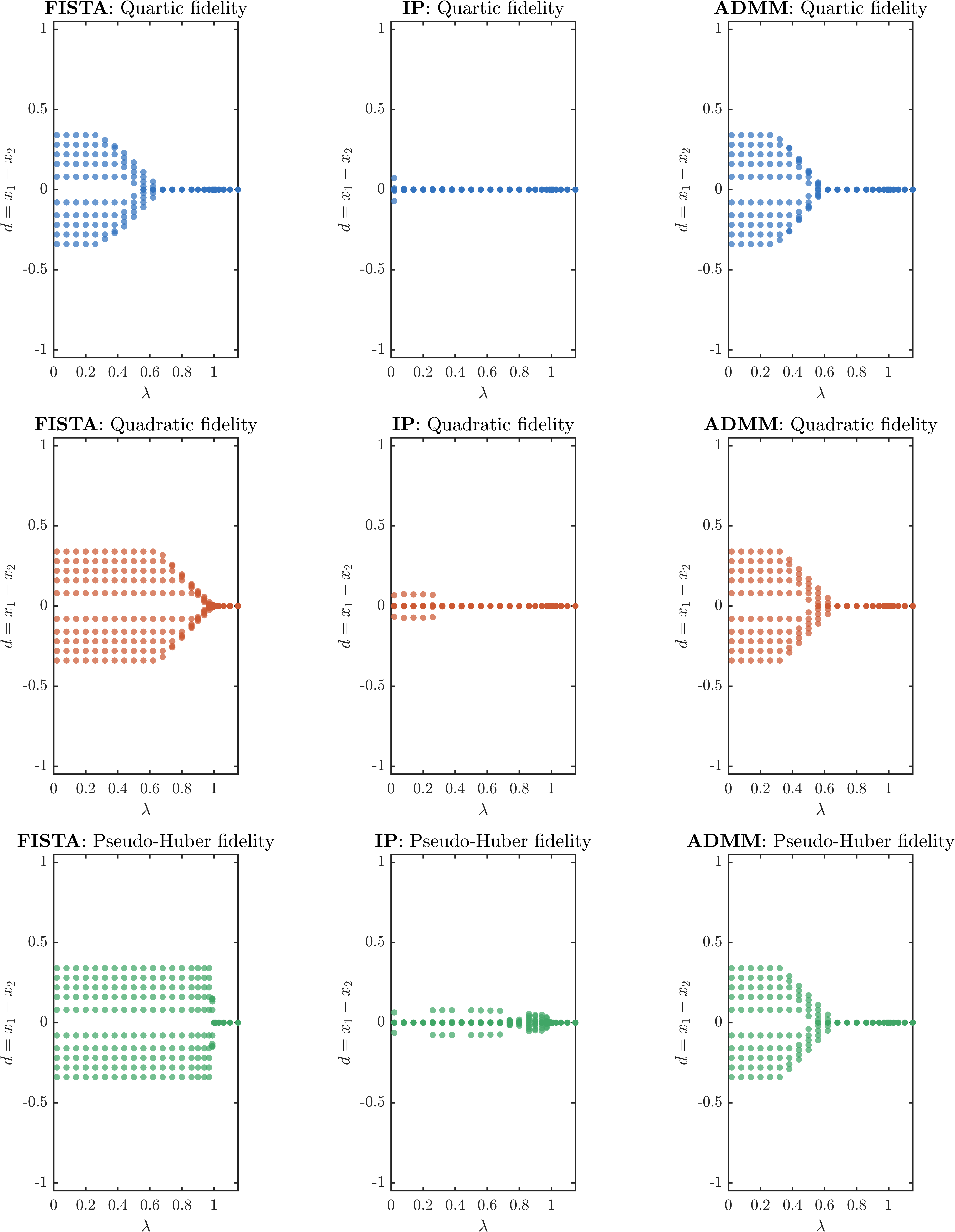}
    \caption{Numerical selections obtained by FISTA, IP,
    and ADMM from the second initialization set
    in the two-dimensional $\ell_1$ example.}
    \label{fig:app_direct_2d_B}
\end{figure}

\begin{figure}[ht]
    \centering
    \includegraphics[width=0.75\textwidth]{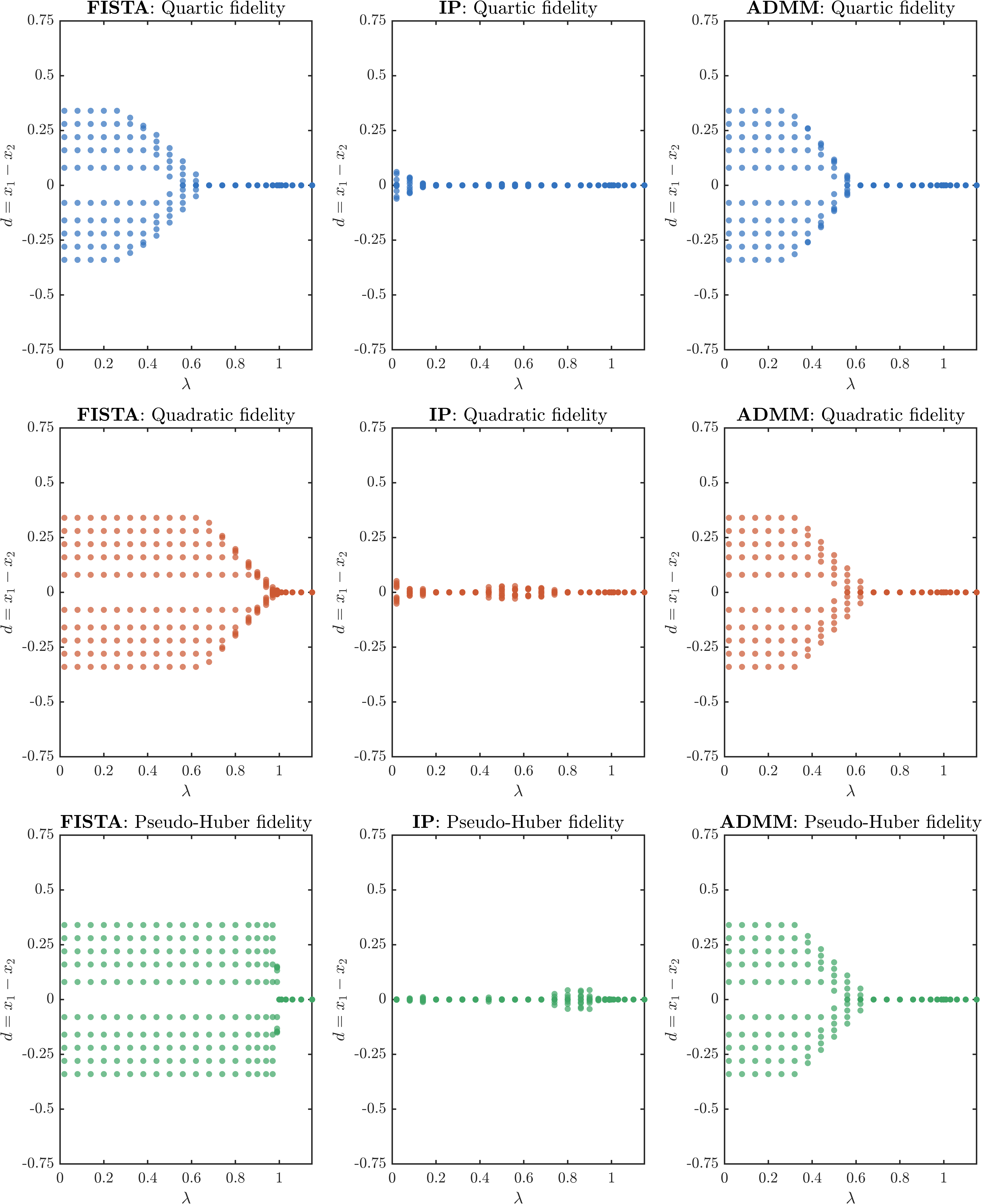}
    \caption{Numerical selections obtained by FISTA, IP,
    and ADMM from the second initialization set
    in the example with nonpolyhedral $\operatorname{gph}\Phi$.}
    \label{fig:app_nonpoly_direct_B}
\end{figure} 


\subsection{Three-dimensional extension of the $\ell_1$ example}
\label{app:three_dimensional}

We extend the two-dimensional $\ell_1$ example to illustrate
how the dimension of $\ker A$ affects the solution fibers.
Let
\[
A=(1,1,1)\in\mathbb R^{1\times3},
\qquad R(x)=\|x\|_1.
\]
Since $V(y)=\min_{Ax=y}\|x\|_1=|y|$, the optimal image
paths $y_j^*(\lambda)$ coincide with those of the
two-dimensional example. However, for $y\geq0$,
\[
\Phi(y)
=
\{x\in\mathbb R_+^3:x_1+x_2+x_3=y\},
\]
which is a two-dimensional simplex for $y>0$
and reduces to $\{0\}$ at $y=0$.

To visualize these fibers, we use coordinates along
an orthonormal basis of $\ker A$:
\[
z_1=\frac{x_1-x_2}{\sqrt2},
\qquad
z_2=\frac{x_1+x_2-2x_3}{\sqrt6}.
\]
Here $y=Ax$ determines the image, while $(z_1,z_2)$
parameterize the two kernel directions.
Thus the nonzero fibers are triangles rather than
the line segments of the two-dimensional example;
see Figure~\ref{fig:app_Phi_3d}.

\begin{figure}[ht]
    \centering
    \includegraphics[width=0.30\textwidth]{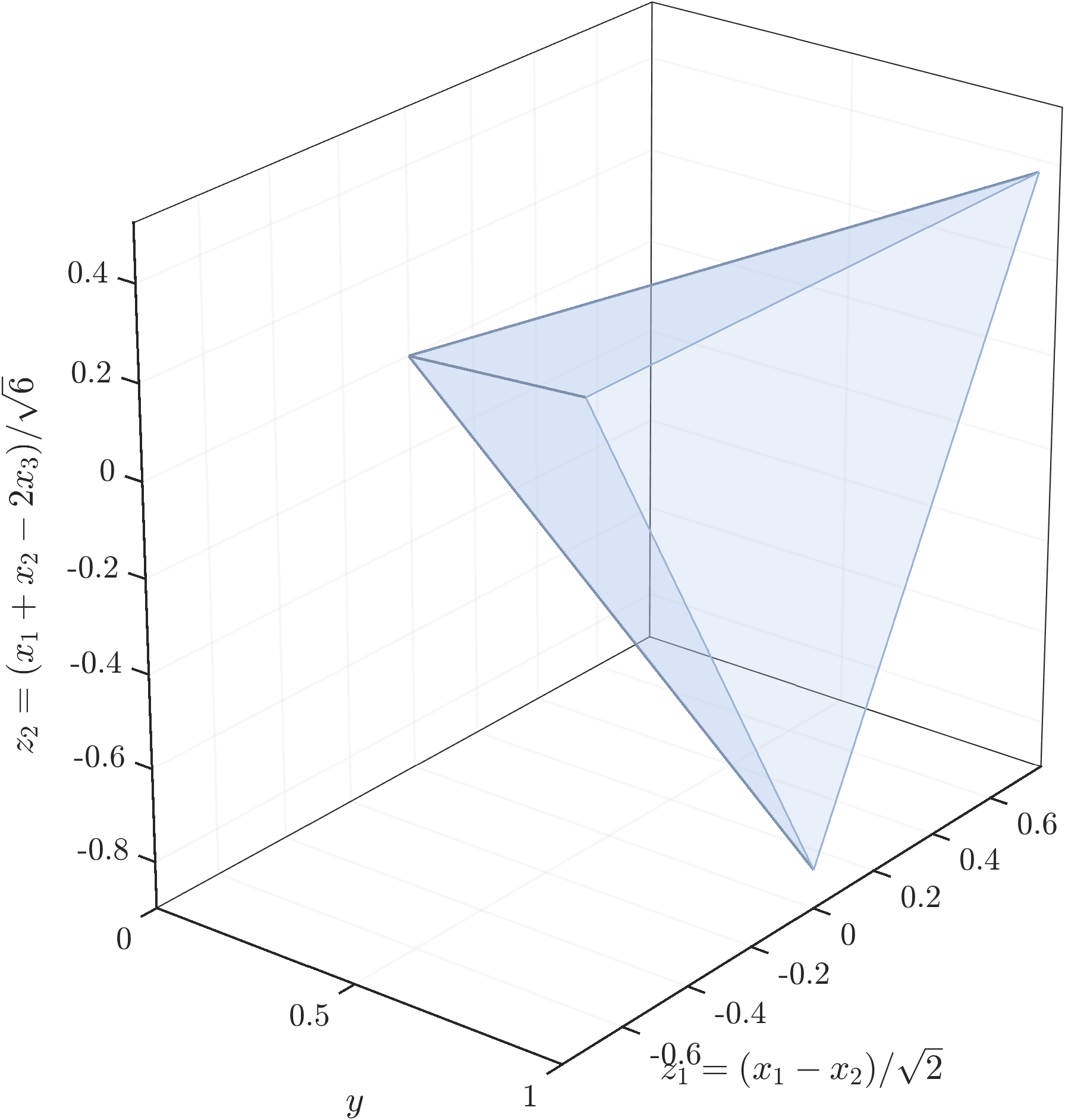}
    \caption{Graph of $\Phi$ in the $(y,z_1,z_2)$ coordinates.
    Each fiber for $y>0$ is a triangle that shrinks
    to the origin as $y\downarrow0$.}
    \label{fig:app_Phi_3d}
\end{figure}

Combining $\Phi$ with the optimal image paths gives
\[
S_j(1,\lambda)
=
\Phi\bigl(y_j^*(\lambda)\bigr)
=
\begin{cases}
\{x\in\mathbb R_+^3:\sum_{i=1}^3x_i=y_j^*(\lambda)\},
& 0<\lambda<1,\\
\{0\}, & \lambda\geq1,
\end{cases}
\]
for $j=1,2,3$. For $0<\lambda<1$, each solution
set is the triangle with vertices
$y_j^*(\lambda)e_i$, $i=1,2,3$.
Figure~\ref{fig:app_S_3d} shows the complete
solution mappings.

\begin{figure}[ht]
    \centering
    \includegraphics[width=0.8\textwidth]{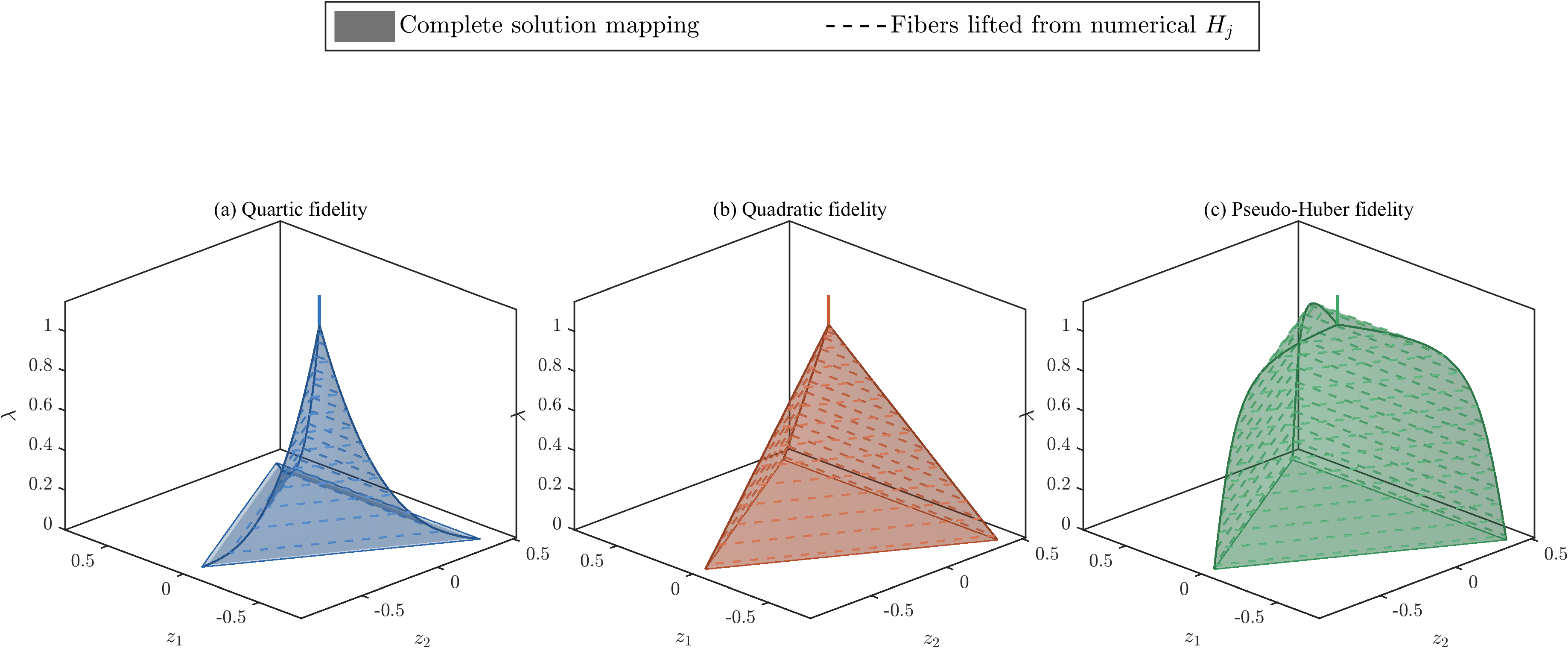}
    \caption{Complete solution mappings in the three-dimensional
    $\ell_1$ example. Shaded regions represent the complete
    solution sets, while dashed triangular fibers are recovered
    from numerically computed optimal images via the
    composite representation.}
    \label{fig:app_S_3d}
\end{figure}

For direct numerical comparisons, we use two initialization
sets. The first contains cross-quadrant and multi-scale points:
\[
X_A^0
=
\begin{pmatrix}
0.35&-0.10&-0.35&0.10&0.70&-0.20&-0.70&0.20&1.40&-0.40\\
0.10&0.25&-0.10&-0.25&0.20&0.50&-0.20&-0.50&0.40&1.00\\
0.05&0.15&-0.05&-0.15&0.10&0.30&-0.10&-0.30&0.20&0.60
\end{pmatrix}.
\]
The second fixes the image and regularizer value:
\[
X_B^0
=
\begin{pmatrix}
0.90&0.05&0.05&0.60&0.60&0.30&0.30&0.10&0.10&1/3\\
0.05&0.90&0.05&0.30&0.10&0.60&0.10&0.60&0.30&1/3\\
0.05&0.05&0.90&0.10&0.30&0.10&0.60&0.30&0.60&1/3
\end{pmatrix}.
\]
Each column represents an initial point, and every
column of $X_B^0$ satisfies $Ax^0=\|x^0\|_1=1$.
Thus $X_B^0$ varies the primal coordinates along
$\ker A$ while preserving the image and regularizer value.

Figures~\ref{fig:app_direct_3d_A} and
\ref{fig:app_direct_3d_B} report the direct numerical
selections obtained by FISTA, IP, and ADMM from
$X_A^0$ and $X_B^0$, respectively.

Figure~\ref{fig:app_slice_3d} further compares the
complete triangular fibers with direct numerical
selections at $\lambda=0.2$.

\begin{figure}[ht]
    \centering
    \includegraphics[width=0.85\textwidth]{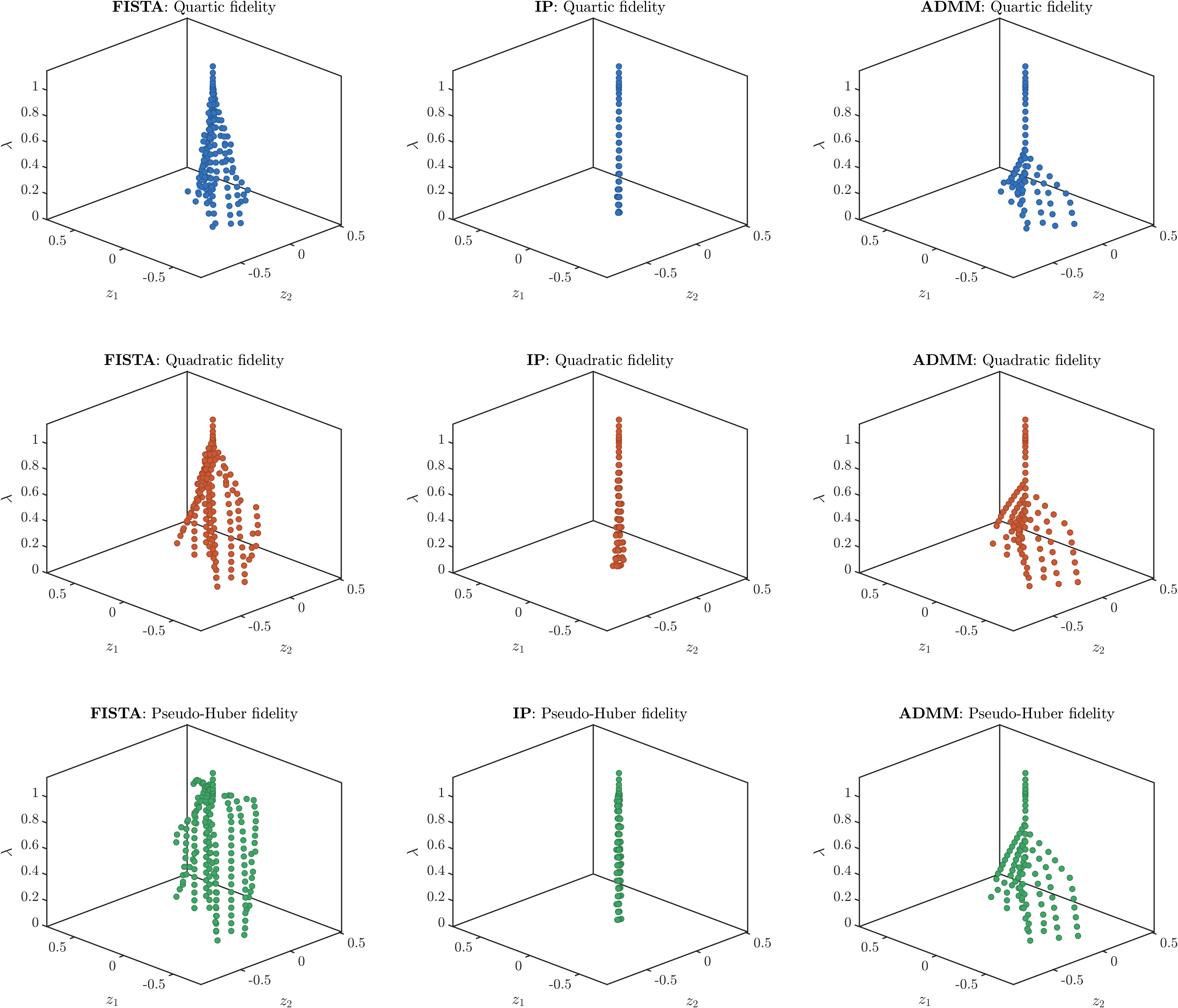}
    \caption{Direct numerical selections obtained by FISTA,
    IP, and ADMM from $X_A^0$ in the three-dimensional
    $\ell_1$ example.}
    \label{fig:app_direct_3d_A}
\end{figure}

\begin{figure}[ht]
    \centering
    \includegraphics[width=0.85\textwidth]{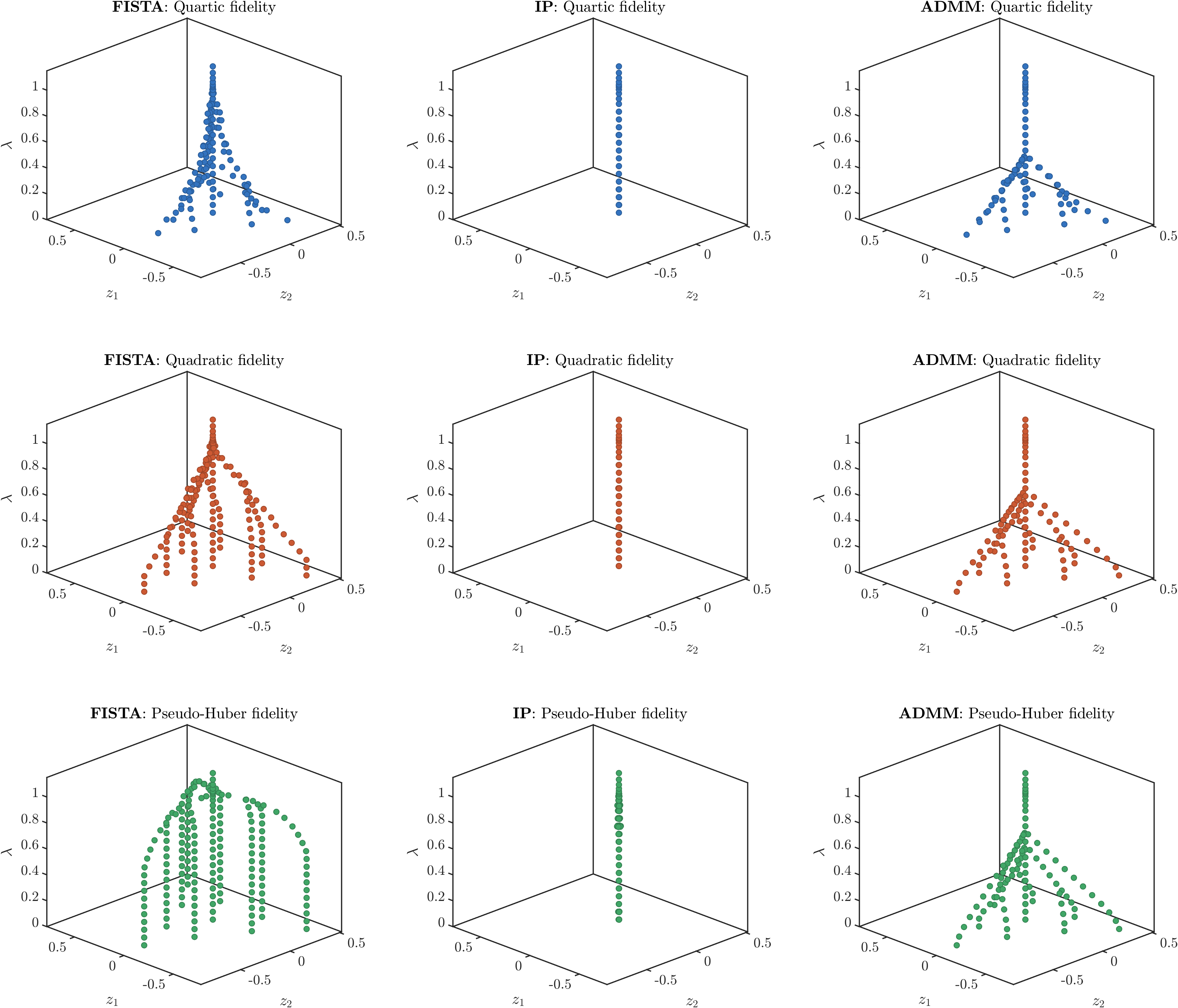}
    \caption{Direct numerical selections obtained by FISTA,
    IP, and ADMM from $X_B^0$ in the three-dimensional
    $\ell_1$ example.}
    \label{fig:app_direct_3d_B}
\end{figure}

\begin{figure}[ht]
    \centering
    \includegraphics[width=0.85\textwidth]{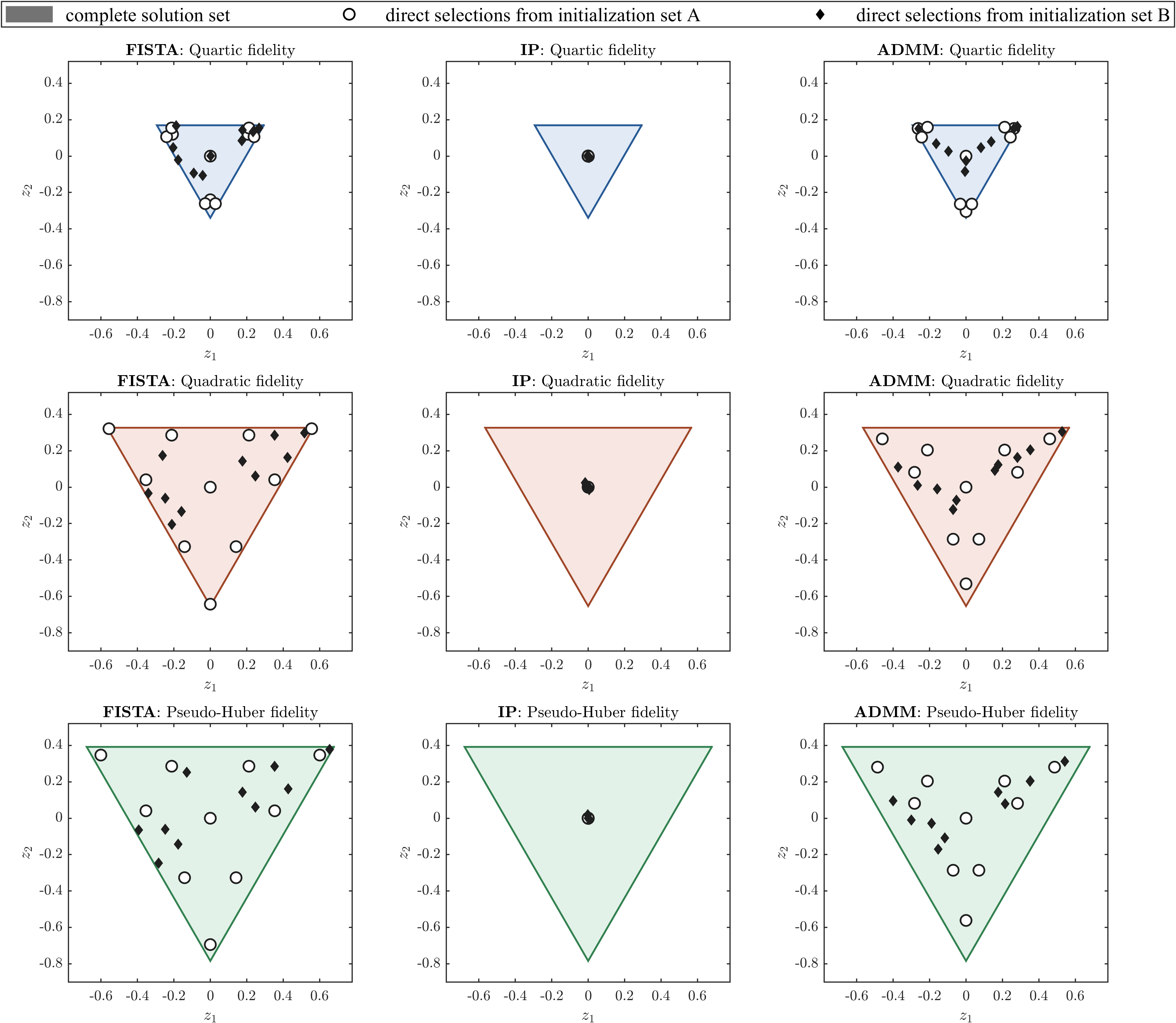}
    \caption{Solution-set cross-sections at $\lambda=0.2$.
    Triangular regions represent the complete fibers
    recovered through the composite representation,
    while points denote solutions returned by FISTA,
    IP, and ADMM from both initialization sets.}
    \label{fig:app_slice_3d}
\end{figure}


\subsection{Total variation reconstruction from sparse observations}
\label{app:tv_numerical} 

We illustrate the composite representation through
one-dimensional signal reconstruction from sparse, noisy
observations using total variation (TV) regularization
\cite{ChambollePock2011,RudinOsherFatemi1992}.
The original problem has $257$ variables, while the
reduced image problem has only $17$ and is strongly convex.
Although the optimal image $y^*$ is unique, the primal
solution set may remain nonunique. The representation
$S(b,\lambda)=\Phi(y^*)$ allows us to recover the complete
solution set and examine how its geometry varies with
the regularization parameter. 

\medskip

\noindent\textbf{Problem formulation and composite representation}

Let $x\in\mathbb R^n$ be a signal and
$1=s_1<\cdots<s_m=n$ the observed coordinates.
Define the first-difference operator
$D\in\mathbb R^{(n-1)\times n}$ and the observation
operator $A\in\mathbb R^{m\times n}$ by
\[
(Dx)_i:=x_{i+1}-x_i,
\qquad
(Ax)_j:=x_{s_j}.
\]
We consider the TV-regularized reconstruction problem
\begin{equation}\label{e-tv-supp-problem}
\min_{x\in\mathbb R^n}
\left\{
\frac12\|Ax-b\|^2+\lambda\|Dx\|_1
\right\},
\qquad \lambda>0.
\end{equation}
The TV regularizer is an instance of the $\ell_1$-type regularizer $R_E(x)=\|Ex\|_1$ studied in Section~\ref{l1-type regularizers} of the main paper, with $E=D$. 

For the generalized basis pursuit value function,
the triangle inequality gives
$V(y)=\min_{Ax=y}\|Dx\|_1
=\sum_{j=1}^{m-1}|y_{j+1}-y_j|$.
Equality holds precisely when $x$ is monotone
between consecutive observed coordinates.
Thus the reduced image problem is
\[
y^*
=
\arg\min_{y\in\mathbb R^m}
\left\{
\frac12\|y-b\|^2
+\lambda\sum_{j=1}^{m-1}|y_{j+1}-y_j|
\right\}.
\]
Its objective is strongly convex, so $H(b,\lambda)=\{y^*\}$.
The composite representation then gives
$S(b,\lambda)=\Phi(y^*)$, where $\Phi(y^*)$
contains all monotone interpolations between
consecutive optimal image values.

\begin{proposition}[Reduced TV value and complete fiber]
\label{prop:tv-fiber-supp}
For every $y\in\mathbb R^m$,
\[
V(y)=\sum_{j=1}^{m-1}|y_{j+1}-y_j|.
\]
Moreover, $x\in\Phi(y)$ if and only if
$x_{s_j},\ldots,x_{s_{j+1}}$ is monotone
from $y_j$ to $y_{j+1}$ for every
$j=1,\ldots,m-1$.
\end{proposition}

\begin{proof}
For any $x$ satisfying $Ax=y$, the triangle
inequality gives
\[
\sum_{k=s_j}^{s_{j+1}-1}|x_{k+1}-x_k|
\geq
|x_{s_{j+1}}-x_{s_j}|
=
|y_{j+1}-y_j|.
\]
Equality holds precisely when the increments
have the same weak sign, including the case
where all increments vanish.
Such monotone interpolations exist on every
interval and attain the lower bound, proving
both claims.
\end{proof}

The fiber geometry can be characterized explicitly
through the following proposition.

\begin{proposition}[Dimension, envelopes, and diameter]
\label{prop:tv-geometry-supp}
For $j=1,\ldots,m-1$, let
\[
q_j:=s_{j+1}-s_j-1,
\qquad
d_j:=y_{j+1}-y_j,
\]
where $q_j$ is the number of unobserved coordinates
between two consecutive observed coordinates
$s_j$ and $s_{j+1}$, and $d_j$ is the difference
between the corresponding image values.

Then
\[
\dim\operatorname{aff}\Phi(y)
=
\sum_{j:d_j\neq0}q_j.
\]
For $1\leq r\leq q_j$, the pointwise bounds are
\[
\begin{aligned}
\min_{x\in\Phi(y)}x_{s_j+r}
&=\min\{y_j,y_{j+1}\},\\
\max_{x\in\Phi(y)}x_{s_j+r}
&=\max\{y_j,y_{j+1}\}.
\end{aligned}
\]
Moreover, the Euclidean diameter is
\begin{equation}\label{e-tv-diameter-supp}
\operatorname{diam}\Phi(y)
=
\left(
\sum_{j=1}^{m-1}
q_j(y_{j+1}-y_j)^2
\right)^{1/2}.
\end{equation}
\end{proposition}

\begin{proof}
For $d_j>0$, the increments
$u_r:=x_{s_j+r}-x_{s_j+r-1}$ satisfy
$u_r\geq0$ and
$\sum_{r=1}^{q_j+1}u_r=d_j$,
forming a simplex of dimension $q_j$.
The same holds for $d_j<0$ after reversing
signs, while $d_j=0$ forces a constant
interpolation. Since interval interiors are
disjoint, the dimensions add.

The pointwise bounds follow by placing the
entire change before or after the selected
coordinate. Moreover, any two monotone
interpolations differ by at most $|d_j|$
at each of the $q_j$ unobserved coordinates.
This bound is attained by two opposite step
interpolations: one makes the entire change
immediately after $s_j$, while the other
makes it immediately before $s_{j+1}$.
Their values differ by exactly $|d_j|$
at every unobserved coordinate.
Since such interpolations can be chosen
independently on each interval, the maximum
squared distance is
$\sum_{j=1}^{m-1}q_jd_j^2$,
which proves \eqref{e-tv-diameter-supp}.
\end{proof}

Since $H(b,\lambda)=\{y^*\}$, these propositions
characterize the complete solution set
$S(b,\lambda)=\Phi(y^*)$, including its dimension,
pointwise bounds, and diameter.
The pointwise bounds describe coordinatewise
envelopes, while the complete fiber must also
satisfy the monotonicity conditions.
In contrast, standard direct numerical solvers
typically return individual minimizers, which
do not by themselves characterize the entire
solution set.

\medskip

\noindent\textbf{Experimental setup}

We set $n=257$, $m=17$, and use equally spaced
sampling indices $s_j=1+16(j-1)$,
$j=1,\ldots,17$.
The planted signal is
\[
x_k^{\mathrm{true}}=
\begin{cases}
0.8, & 1\leq k\leq64,\\
-0.4, & 65\leq k\leq128,\\
1.1, & 129\leq k\leq192,\\
0.2, & 193\leq k\leq257.
\end{cases}
\]
We generate noisy observations
$b=Ax^{\mathrm{true}}+0.08\,\xi$, where
$\xi\sim\mathcal N(0,I_{17})$, using MATLAB's
random number generator with seed $20261010$.
The observations are kept fixed for all values
of $\lambda$.
The reduced problem is solved independently for
$25$ logarithmically spaced values of $\lambda$
in $[10^{-3},10]$, without warm starts.

We solve the reduced TV problem using ADMM.
Introducing $z=D_my$, where
$D_m\in\mathbb R^{16\times17}$ is the
first-difference operator on $\mathbb R^{17}$,
the iterations with penalty parameter $\rho=1$
are
\begin{align*}
y^{\ell+1}
&=(I+\rho D_m^\top D_m)^{-1}
\bigl(b+\rho D_m^\top(z^\ell-u^\ell)\bigr),\\
z^{\ell+1}
&=\operatorname{soft}_{\lambda/\rho}
\bigl(D_my^{\ell+1}+u^\ell\bigr),\\
u^{\ell+1}
&=u^\ell+D_my^{\ell+1}-z^{\ell+1},
\end{align*}
where $\operatorname{soft}_{\tau}$ denotes
componentwise soft-thresholding.
For each $\lambda$, we initialize
$y^0=b$, $z^0=D_mb$, and $u^0=0$.
We stop when both the primal residual
$\|D_my^\ell-z^\ell\|$ and the dual residual
$\rho\|D_m^\top(z^\ell-z^{\ell-1})\|$
are below $10^{-10}$.
All $25$ problems satisfy these criteria,
with maximum final residuals
$9.84\times10^{-11}$ and
$9.55\times10^{-11}$, respectively.
The iteration counts range from $57$ to $563$.

Figure~\ref{fig:tv_complete_fibers_supp}
illustrates the optimal images and their
corresponding solution fibers at four values
of $\lambda$. The orange open circles
represent the unique optimal image values $y^*\in\mathbb R^{17}$
at the observed coordinates, while the
shaded bands show the exact pointwise ranges
of $\Phi(y^*)$. The blue curve is one optimal
reconstruction obtained by linear interpolation.
The bands describe coordinatewise bounds
rather than the complete fiber, whose elements
must also satisfy the monotonicity conditions
in Proposition~\ref{prop:tv-fiber-supp}.

\begin{figure}[ht]
\centering
\includegraphics[width=0.9\textwidth]
{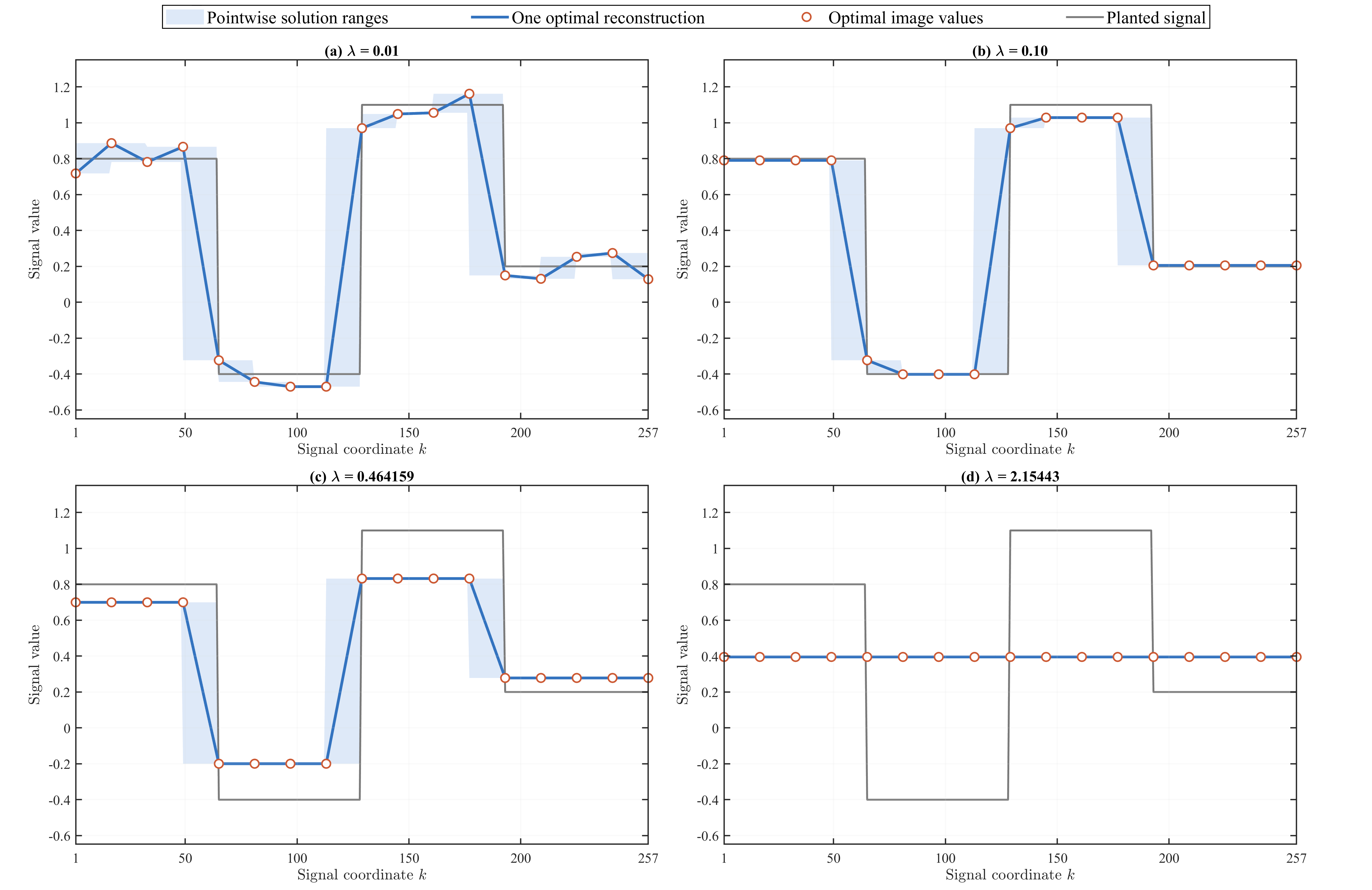}
\caption{Optimal images and solution fibers in TV
reconstruction from sparse observations.
Orange open circles indicate the optimal image
values at the observed coordinates, and the
gray curve shows the planted signal.
The blue curve represents one optimal reconstruction,
while the shaded bands show the exact pointwise
ranges of the complete solution sets.
Admissible reconstructions must also be monotone
between consecutive observed coordinates.
The horizontal axis denotes the signal coordinate
$k=1,\ldots,257$.}
\label{fig:tv_complete_fibers_supp}
\end{figure}

Figure~\ref{fig:tv_diameter_supp} shows the
diameter of the complete solution set as a
function of $\lambda$, computed from
\eqref{e-tv-diameter-supp}.
In this experiment, increasing $\lambda$
progressively fuses neighboring optimal image
values, reducing the dimension and diameter
of the corresponding solution fibers.

\begin{figure}[ht]
\centering
\includegraphics[width=0.75\textwidth]
{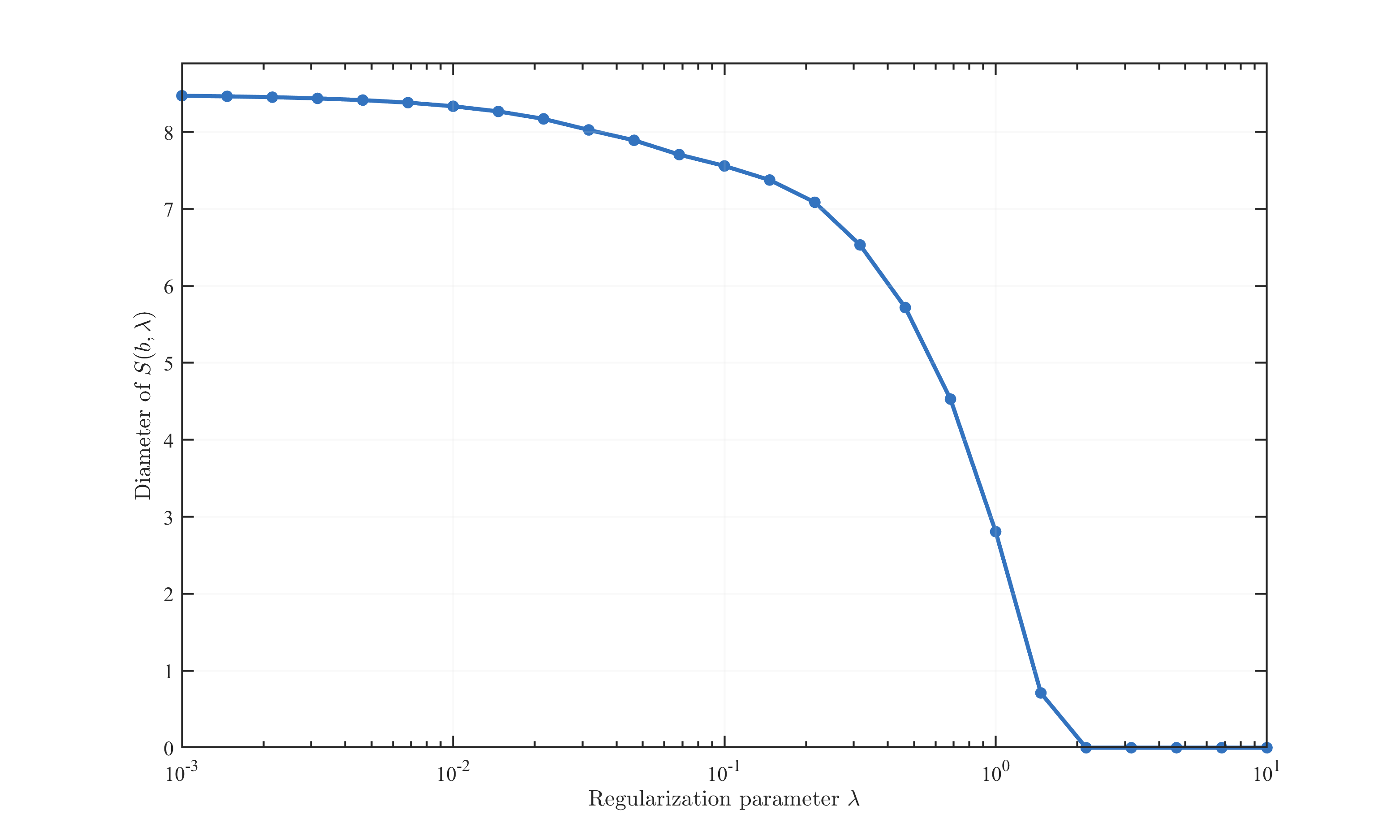}
\caption{Exact diameter of the complete TV
solution set as a function of the regularization
parameter $\lambda$, computed
from~\eqref{e-tv-diameter-supp}.}
\label{fig:tv_diameter_supp}
\end{figure}

\medskip
\noindent\textbf{Relation to generalized-lasso path methods}

Generalized-lasso path algorithms, related to LARS for the classical lasso, exploit the piecewise-linear
structure of solution paths as $\lambda$ varies \cite{TibshiraniTaylor2011}.
Our objective is complementary: for each $\lambda$,
we recover the complete primal solution set from the
unique optimal image $y^*$ and the fiber mapping
$\Phi(y^*)$, including its dimension, pointwise bounds,
and diameter. A path algorithm could replace ADMM in
computing $y^*$, but the image path alone does not
characterize the entire fiber.

Once $y^*$ is computed, no further optimization in
$\mathbb R^{257}$ is needed. The complete solution set
is characterized by the monotonicity conditions in
Proposition~\ref{prop:tv-fiber-supp}, while its dimension,
pointwise bounds, and diameter follow from
Proposition~\ref{prop:tv-geometry-supp}.
Table~\ref{tab:tv_results} reports the geometry of
the four representative solution fibers shown in
Figure~\ref{fig:tv_complete_fibers_supp}.
These parameter values illustrate the reduction in
the dimension and diameter of the complete solution
set, culminating in a singleton.
The affine dimensions are computed using a
$10^{-8}$ tolerance for detecting nonzero adjacent
differences. Since each interval contains 15
unobserved coordinates, the reported dimensions
correspond to 15, 5, 3, and 0 nonzero adjacent
differences in the respective optimal images.

\begin{table}[ht]
\centering
\caption{Dimensions and diameters of the four optimal
TV fibers illustrated in
Figure~\ref{fig:tv_complete_fibers_supp}.}
\label{tab:tv_results}
\begin{tabular}{@{}rrr@{}}
\hline
$\lambda$
& $\operatorname{diam}S(b,\lambda)$
& $\dim\operatorname{aff}S(b,\lambda)$\\
\hline
$0.01$     & $8.33398$ & $225$\\
$0.10$     & $7.55860$ & $75$\\
$0.464159$ & $5.71840$ & $45$\\
$2.15443$  & $0$       & $0$\\
\hline
\end{tabular}
\end{table}

The solution set becomes a singleton when the optimal
image is constant. Indeed, writing
$\bar b:=m^{-1}\sum_{j=1}^m b_j$, the reduced optimality
condition gives the threshold
\[
\lambda_{\mathrm{crit}}
=
\max_{1\leq k<m}
\left|\sum_{j=1}^k(b_j-\bar b)\right|.
\]
For the generated observations,
$\lambda_{\mathrm{crit}}\approx1.68159$.
Consequently, for every
$\lambda\geq\lambda_{\mathrm{crit}}$,
the optimal image is constant and
\[
S(b,\lambda)=\{\bar b\mathbf{1}_{257}\}.
\]
Thus the zero diameter at $\lambda=2.15443$
holds exactly, independently of numerical tolerance.

This experiment illustrates how the composite
representation characterizes all minimizers and
quantifies their nonuniqueness, rather than returning
only individual numerical selections.

\end{document}